\documentclass[12pt, reqno]{amsart}
\usepackage{mathrsfs}
\usepackage{amsfonts}
\usepackage[sorted,compressed-cites,sorted-cites,initials]{amsrefs}
\usepackage[centertags]{amsmath}
\usepackage{amssymb,array}
\usepackage{amsthm,youngtab}
\usepackage{graphicx}
\usepackage{caption}
\usepackage{ytableau}
\usepackage{enumerate,enumitem}
\SetLabelAlign{center}{\hfill #1\hfill}
\usepackage[textwidth=16cm, hmarginratio=1:1]{geometry}
\usepackage{tikz}
\usetikzlibrary{arrows,automata, positioning}
\usepackage{rotating}
\usepackage{mathdots}
\usepackage[all,cmtip]{xy}
\usepackage[titletoc, title]{appendix}
\usepackage{amscd}
\usepackage[colorlinks]{hyperref}
\usepackage{float}
\hypersetup{
bookmarksnumbered,
pdfstartview={FitH},
breaklinks=true,
linkcolor=blue,
urlcolor=blue,
citecolor=blue,
bookmarksdepth=2
}
\usetikzlibrary{mindmap,trees}

\theoremstyle{plain}
   \newtheorem{theorem}{Theorem}[section]
   
   \newtheorem{lemma}[theorem]{Lemma}

\theoremstyle{definition}
   \newtheorem{definition}[theorem]{Definition}
   \newtheorem{example}[theorem]{Example}

   \newtheorem{remark}[theorem]{Remark}

\numberwithin{equation}{section}

\newlength{\mysizetiny}
\newlength{\mysizesmall}
\newlength{\mysize}
\newlength{\mysizelarge}
\begin{document}

\title{Hernandez-Leclerc modules and snake graphs}
\author{Bing Duan, Jian-Rong Li and Yan-Feng Luo}
\address{Bing Duan, School of mathematics and statistics, Lanzhou University, 730000, P. R. China; School of physical science and technology, Lanzhou University, 730000, P. R. China.} 
\email{duanb@lzu.edu.cn}
\address{Jian-Rong Li, Institute of Mathematics and Scientific Computing, University of Graz, Graz 8010, Austria.} 
\email{lijr07@gmail.com}
\address{Yan-Feng Luo, School of mathematics and statistics, Lanzhou University, 730000, P. R. China.} 
\email{luoyf@lzu.edu.cn} 
\date{}

\begin{abstract}
In 2010, Hernandez and Leclerc studied connections between representations of quantum affine algebras and cluster algebras. In 2019, Brito and Chari defined a family of modules over quantum affine algebras, called Hernandez-Leclerc modules.  We characterize the highest $\ell$-weight monomials of Hernandez-Leclerc modules. We give a non-recursive formula for $q$-characters of Hernandez-Leclerc modules using snake graphs, which involves an explicit formula for $F$-polynomials. We also give a new recursive formula for $q$-characters of Hernandez-Leclerc modules. 
\end{abstract}


\maketitle

\tableofcontents

\section{Introduction}

Quantum groups were introduced independently by Drinfeld \cite{D85,D87} and Jimbo \cite{Jim85}. Let $\mathfrak{g}$ be a simple Lie algebra over $\mathbb{C}$ and let $U_{q}(\widehat{\mathfrak{g}})$ be the corresponding untwisted quantum affine algebra with quantum parameter $q\in \mathbb{C}^\times$ not a root of unity. Denote by $\mathscr{C}$ the category of finite dimensional representations of $U_{q}(\widehat{\mathfrak{g}})$. It is well-known that $\mathscr{C}$ is not semisimple but an abelian tensor category \cite{HL10}. The isomorphism classes of finite-dimensional simple $U_{q}(\widehat{\mathfrak{g}})$-modules can be parametrized by Drinfeld polynomials in \cite{CP91,CP94,CP95}. Equivalently, each simple $U_{q}(\widehat{\mathfrak{g}})$-module can also be parametrized by the highest dominant monomial of its $q$-character \cite{FR98}.

Cluster algebras were introduced by Fomin and Zelevinsky in their seminal work \cite{FZ02}. A cluster algebra is a commutative ring with a set of distinguished generators called cluster variables which are defined through iterative processes known as mutations. 
 
A finite-dimensional $U_{q}(\widehat{\mathfrak{g}})$-module is said to be \textit{prime} if it admits no nontrivial tensor factorization \cite{CP97}. A simple $U_{q}(\widehat{\mathfrak{g}})$-module is said to be \textit{real} if its tensor square is also simple \cite{L03}. In \cite{HL10}, Hernandez and Leclerc introduced the notion of monoidal categorification of cluster algebras.  An abelian tensor category is said to be a monoidal categorification of a cluster algebra if the Grothendieck ring of the category is isomorphic to the cluster algebra and the classes of the real (respectively, real prime) simple modules correspond to cluster monomials (respectively, cluster variables).

For each $\ell\in \mathbb{Z}_{\geq 0}$, Hernandez and Leclerc \cite{HL10} introduced a full monoidal subcategory $\mathscr{C}_\ell$ of $\mathscr{C}$ whose objects are characterized by certain restrictions on the roots of the Drinfeld polynomials of their composition factors. 

\subsection{Highest $\ell$-weight monomials of Hernandez-Leclerc modules}
As a generalization of $\mathscr{C}_1$ \cite{HL10,HL13}, in \cite{BC19} Brito and Chari introduced a subcategory $\mathscr{C}_\xi$ and a quiver $Q_\xi$ depending on a choice of a height function $\xi$. These real prime simple $U_{q}(\widehat{\mathfrak{g}})$-modules in $\mathscr{C}_\xi$ are called Hernandez-Leclerc modules by Brito and Chari. Brito and Chari proved that $\mathscr{C}_\xi$ is a monoidal categorification of the cluster algebra $\mathcal{A}({\bf \widetilde{x}},Q_\xi)$ with coefficients of type $A$ \cite[Section 1.3]{BC19}. Hernandez-Leclerc modules are in bijection with cluster variables and frozen variables, among the initial cluster variables and frozen variables correspond to some Kirillov-Reshetikhin modules. Let $\mathcal{K}_0(\xi)$ be the Grothendieck ring of $\mathscr{C}_\xi$. Denote by $\iota$ the algebraic isomorphism $\mathcal{A}({\bf \widetilde{x}},Q_\xi) \to \mathcal{K}_0(\xi)$ \cite{BC19}. Moreover, Brito and Chari gave a  non-recursion for $q$-characters of Hernandez-Leclerc modules in term of the known $q$-characters of initial simple modules \cite[Proposition in Section 2.5]{BC19}. 

In this paper, we use snake graphs to study $q$-characters of Hernandez-Leclec modules. Let $I=\{1,2,\ldots,n\}$. We first give a combinatorial characterization of the highest $\ell$-weight monomials of Hernandez-Leclerc modules. 

\begin{theorem}[{Theorem \ref{HL defintion equivalent}}]
An Hernandez-Leclerc module corresponding to a cluster variable (excluding frozen variables) is a simple $U_{q}(\widehat{\mathfrak{g}})$-module with the highest $\ell$-weight monomial
\begin{align*} 
Y_{i_1,a_1}Y_{i_2,a_2}\ldots Y_{i_k,a_k},
\end{align*}
where $k\in\mathbb{Z}_{\geq 1}$, $i_j\in I$, $a_j\in \mathbb{Z}$ for $j=1,2,\ldots,k$, and
\begin{itemize}
\item[{\rm (i)}] $i_1<i_2<\cdots<i_k$,
\item[{\rm (ii)}] $(a_{j}-a_{j-1})(a_{j+1}-a_{j})<0$ for $2\leq j\leq k-1$,
\item[{\rm (iii)}] $|a_{j}-a_{j-1}|=i_{j}-i_{j-1}+2$ for $2\leq j\leq k$.
\end{itemize}
\end{theorem}

\subsection{Hernandez-Leclerc modules, snake graphs, and $q$-characters}
In \cite{MY12a,MY12b}, Mukhin and Young gave a purely combinatorial $q$-character formula for snake modules of types $A_n$ and $B_n$ via path descriptions. A completely different approach to compute $q$-characters of Kirillov-Reshetikhin modules was  developed by Hernandez and Leclerc \cite{HL16}. In fact, Hernandez and Leclerc proposed a geometric $q$-character formula conjecture for real simple modules, which implies that the truncated $q$-character formula will play an important role similar to the cluster character formula. Recently, Duan and Schiffler in \cite{DS20} proved the geometric $q$-character formula for snake modules of types $A_n$ and $B_n$, and characterized the associated general kernel. In \cite{CDFL20}, the authors gave an explicit $q$-character formula of simple $U_q(\widehat{\mathfrak{sl}}_n)$-modules. The formula involves Kazhdan-Lusztig polynomials \cite{KL79} which are hard to compute when the length of the highest $\ell$-weight monomial of a simple module is large.

Our aim is to give a non-recursive formula for $q$-characters of Hernandez-Leclec modules using snake graphs. In \cite{FST08}, Fomin, Shapiro, and Thurston introduced an important class of cluster algebras from surfaces with or without punctures. Each cluster variable (cluster) corresponds to a tagged arc (ideal triangulation) in the surface. In \cite{MSW11,MSW13}, Musiker, Schiffler, and Williams constructed a combinatorial object associated to a non-initial arc, called a labeled snake graph, to compute the Laurent expansion of any non-initial cluster variable in a cluster algebra from a surface. Later Canakci and Schiffler studied abstract snake graphs in a series of papers \cite{CS13,CS15,CS19} and established interesting connections among cluster variables, snake graphs, and continued fractions \cite{CS18,CS20}. In particular, identities in the cluster algebra have been expressed in terms of snake graphs. Rabideau used continued fractions to give a combinatorial formula for $F$-polynomials \cite{Rab18} in cluster algebras with principal coefficients from surfaces. Rabideau and Schiffler proved the constant numerator conjecture for Markov numbers via snake graphs and continued fractions \cite{RS20}.

Given a height function $\xi$, we construct a unique labeled snake graph for a non-initial Hernandez-Leclerc module. Our idea is summarized by the diagram in Figure \ref{four correspondences}. 
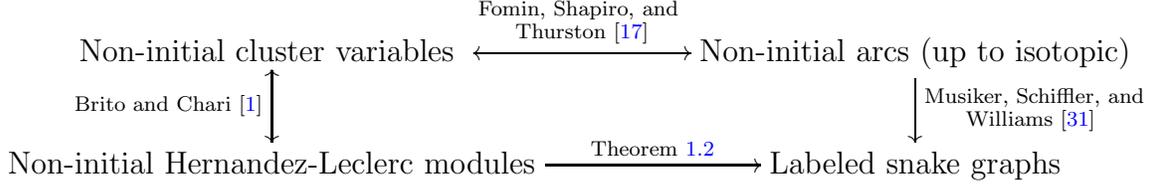
\begin{figure}
\centerline{
\xymatrix{
\text{Non-initial cluster variables } \ar@{<->}[rr]^-{ \substack{ \text{Fomin, Shapiro, and } \\  \text{Thurston \cite{FST08}}}}  && \text{Non-initial arcs (up to isotopic)}  \ar@{->}[d]^-{ \substack{\text{Musiker, Schiffler, and}  \\ \text{Williams \cite{MSW11} }}}   \\
\text{Non-initial Hernandez-Leclerc modules} \ar@{<->}[u]^-{\text{Brito and Chari \cite{BC19}}}  \ar@{->}[rr]^-{\substack{\text{Theorem \ref{theorem between HL-modules and snake graphs}}}}  && \text{Labeled snake graphs}  \\
} }
\caption{Cluster variables, arcs, Hernandez-Leclerc modules, and labeled snake graphs.}\label{four correspondences}
\end{figure}

Fix a simple root system $\{\alpha_i\mid i\in I\}$ of type $A_n$, let $\alpha_{i,j}=\alpha_i + \cdots + \alpha_j$ be a positive root for $1\leq i\leq j\leq n$. Then the set of all cluster variables in $\mathcal{A}({\bf \widetilde{x}},Q_\xi)$ is $\{ x[-\alpha_i] \mid i\in I\} \cup \{x[\alpha_{i,j}] \mid 1\leq i\leq j\leq n\}$. Denote by ${\rm Match}(\mathcal{G})$ the set of all perfect matchings of a snake graph $\mathcal{G}$. For any $P\in {\rm Match}(\mathcal{G})$, let $x(P)$ (respectively, $y(P)$) be the weight  (respectively, height) monomial in the sense of Musiker, Schiffler, and Williams \cite{MSW11,MSW13}. 

By applying the theory of cluster algebras, we give a $q$-character formula of an arbitrary Hernandez-Leclerc module corresponding to a non-initial cluster variable (excluding frozen variables)  by perfect matchings of snake graphs.

\begin{theorem}[{Theorems \ref{snake graph formula} and \ref{theorem313}}] \label{theorem between HL-modules and snake graphs}
Let $\mathcal{G}$ be the labeled snake graph associated to $x[\alpha_{i,j}]$. Then
\begin{align*}
x[\alpha_{i,j}]  = \frac{1}{ \prod_{\ell=i}^j x_\ell}  \left(\frac{\sum_{P\in {\rm Match}(\mathcal{G})}  x(P) y(P) }{\bigoplus_{P\in {\rm Match}(\mathcal{G})} y(P)}\right),
\end{align*}
where the sign $\oplus$ appearing in the denominator refers to the addition of a tropical semifield. In particular, 
\begin{align*}
\chi_q(\iota(x[\alpha_{i,j}])) = \frac{1}{ \prod_{\ell=i}^j \chi_q(\iota(x_\ell))} \chi_q\left(\iota(\frac{\sum_{P\in {\rm Match}(\mathcal{G})}  x(P) y(P) }{\bigoplus_{P\in {\rm Match}(\mathcal{G})}  y(P)})\right).
\end{align*}
There is a unique perfect matching $P\in {\rm Match}(\mathcal{G})$ such that the highest or lowest $\ell$-weight monomial in $\chi_q(\iota(x[\alpha_{i,j}]))$ occurs in 
\begin{align*}
\frac{1}{ \prod_{\ell=i}^j \chi_q(\iota(x_\ell))} \chi_q\left(\iota(\frac{x(P) y(P) }{\bigoplus_{P\in {\rm Match}(\mathcal{G})} y(P)})\right).
\end{align*}
\end{theorem}

In \cite[Proposition in Section 2.5]{BC19}, Brito and Chari gave a non-recursive formula of $x[\alpha_{i,j}]$ using two sets  $\Gamma_{i,j}$ and $\Gamma'_{i,j}$. The set $\Gamma_{i,j}$ consists of $(0,1)$ sequences with $(j-i+2)$ length subject to four conditions, and  $\Gamma'_{i,j}$ is determined by $\Gamma_{i,j}$ consisting of $(-1,0,1)$ sequences with the same length, see Section \ref{relation Brito-Chari result}.

On the other hand, we associate to $x[\alpha_{i,j}]$ perfect matchings of the corresponding labeled snake graph. Combining with Rabideau's formula for $F$-polynomials \cite[Theorem 3.4]{Rab18}, Theorem \ref{dual case of F-polynomial} gives an explicit formula for $F$-polynomials. Moreover, we give the evaluation of $F$-polynomials at a tropical semifield, see Lemma \ref{Lemma36}. 

\subsection{A recursive formula for Hernandez-Leclerc modules}

We give a new recursive formula for Hernandez-Leclerc modules by an induction on the length of the highest $\ell$-weight monomials of Hernandez-Leclerc modules.

By Brito and Chari's results \cite[Theorem 1, Corollary in Section 1.3]{BC19}, the class of each Hernandez-Leclerc module corresponds to a cluster variable $x[\alpha_{i,j}]$, where $j-1$ is a source or sink, see the paragraph after Remark \ref{note of our formula}. 
\begin{theorem} [{Theorem \ref{theorem41}}]
Let $j\in I$ be a source or sink vertex and $j>i\in I$. Then 
\begin{multline}\label{formula for HL by length}
x[\alpha_{i,j}] x[\alpha_{j+1,j+1}] =  x[\alpha_{i,j+1}] \\
+ x[\alpha_{i,\max\{i-1,j_\bullet-1\}}]^{1-\delta_{i,j_\bullet}} {x'}_{\max\{i, j_\bullet\}}^{ \min\{1,(1-\delta_{j_\bullet,i_\bullet})d_{j_\bullet-1}+\delta_{j_\bullet,i}\}} x[-\alpha_{j+2}]^{d_{j+1}} {x'}^{1-d_{j+1}}_{j+2},
\end{multline}
where $\delta_{i,j}$ is the Kronecker symbol, $d_j=\delta_{j,j_\diamond}$, and $j_\bullet$ is the maximum source or sink vertex strictly less than $j$ in $Q_\xi$.
\end{theorem} 
Comparing with a recursive formula given by Brito and Chari in \cite[Proposition in Section 1.5]{BC19}, our recursive formula needs one more step of mutation for a height function $\xi$ such that $j$ is source or sink, see Remark \ref{note of our formula}~(2).

The formula (\ref{formula for HL by length}) provides a possibility of a combinatorial path formula in which we allow overlapping paths, generalizing Mukhin and Young's path formula for snake modules \cite{MY12a,MY12b}.

The content of this paper is outlined as follows. In Section \ref{preface}, we recall some background on cluster algebras of geometric type, Hernandez-Leclerc modules, snake graphs, and $F$-polynomials. In Section \ref{The definition of HL-modules}, we characterize the highest $\ell$-weight monomials of Hernandez-Leclerc modules and give a $q$-character formula in term of perfect matchings of snake graphs. In Section \ref{a recursive formula}, a new recursive formula for $q$-characters of Hernandez-Leclerc modules is introduced.

\section{Preliminary} \label{preface}
\subsection{Cluster algebras}
We recall the definition of cluster algebras \cite{FZ02,FZ07}. 

Let $\mathcal{F}$ be the field of rational functions in $n$ independent variables over $\mathbb{Q}\mathbb{P}$, where $\mathbb{P}$ is a semifield. For $u_1,u_2,\ldots,u_r$ in $\mathbb{P}$, we denote by $F|_{\mathbb{P}}(u_1,\ldots,u_r)$ the evaluation of a subtraction-free rational expression $F$ at $u_1,\ldots,u_r$. 

A \textit{tropical semifiled} $(\mathbb{P}, \oplus, \cdot)$, where $\mathbb{P}=\text{Trop}(u_1,u_2,\ldots,u_r)$, is an abelian multiplicative group freely generated by $u_1,u_2,\ldots,u_r$, the (auxiliary) addition $\oplus$ is defined by
\[
\prod_{j} u^{a_j}_j \oplus \prod_{j} u^{b_j}_j = \prod_{j} u^{\min(a_j,b_j)}_j.
\]
An important choice for coefficients or frozen variables in a cluster algebra is the tropical semifield, in this case, the corresponding cluster algebra is said to be of \textit{geometric type}. In this paper, we pay attention to  skew-symmetric cluster algebras of geometric type. 

Let $Q$ be a finite quiver without loops or 2-cycles. Assume without loss of generality that the vertex set of $Q$ is $\{1,\ldots,m\}$, among vertices $1,\ldots,n$ ($n\leq m$) are mutable vertices and vertices $n+1,\ldots,m$ are frozen vertices. Mutating $Q$ at a mutable vertex $k$,  one obtains a new quiver $\mu_k(Q)$ defined as follows: 
\begin{itemize}
\item[(i)] add a new arrow $i \to j$ for each pair of arrows $i \to k \to j$, excluding that both $i$ and $j$ are frozen vertices;
\item[(ii)] reverse the orientation of each arrow incident to $k$; and 
\item[(iii)] remove all maximal pairwise disjoint 2-cycles.
\end{itemize}

\begin{definition}(Seeds)
A seed in $\mathcal{F}$ is a pair $({\bf \widetilde{x}},Q)$, where 
\begin{itemize}
\item[(i)] ${\bf \widetilde{x}}=\{x_1,\ldots,x_n, x_{n+1},\ldots,x_m\}$ is a free generating set of $\mathcal{F}$, called an \textit{extended cluster}. The subset ${\bf x}=\{x_1,\ldots,x_n\}$ is called a \textit{cluster} and each element in ${\bf x}$ is called a \textit{cluster variable}; $\{x_{n+1},\ldots,x_m\}$ is a set of elements in $\mathbb{P}$, called a \textit{coefficient tuple}, among them each element is called a \textit{coefficient} or \textit{frozen variable}; and
\item[(ii)] Q is a quiver as above. 
\end{itemize}
\end{definition}


\begin{definition}(Seed mutations).
Let $({\bf \widetilde{x}},Q)$ be a seed. The seed mutation $\mu_k$ in direction $k\in \{1,\ldots,n\}$ transforms $({\bf \widetilde{x}},Q)$ into the seed $({\bf \widetilde{x}'},Q')$ defined as follows.
\begin{itemize}
\item[(i)] The extended cluster ${\bf \widetilde{x}'}=\{x'_1,\ldots,x'_m\}$ is defined by $x'_j=x_j$ for $j\neq k$, whereas $x'_k$ is defined by the \textit{exchange relation}
\[
x'_k= \cfrac{ \prod\limits_{i: i\to k} x_i + \prod\limits_{j: k \to j} x_j }{x_k};
\]
\item[(ii)] $Q'=\mu_k(Q)$ defined as above.
\end{itemize}
\end{definition}

A \textit{cluster algebra} is a $\mathbb{ZP}$-subalgebra of $\mathcal{F}$ generated by all cluster variables obtained from seed mutations.

Cluster algebras of finite type are classified by the Dynkin diagrams \cite{FZ03}, that is, there is an exchange matrix such that the Cartan counterpart of its principle part is one of Cartan matrices of finite type. Under this correspondence, cluster variables in a cluster algebra of finite type are in bijection with almost positive roots in the root system associated to the corresponding Cartan matrix.

\begin{theorem}[{\cite[Theorem 1.9]{FZ03}}] \label{cluster variables bijection with almost positive roots}
In a cluster algebra of finite type, the cluster variable $x[\alpha]$ for $\alpha=\sum_{i\in I} a_i \alpha_i$ is expressed in term of the initial cluster ${\bf x}_0$ as 
\[
x[\alpha]= \frac{P_\alpha({\bf x}_0)}{\prod_{i\in I} x^{a_i}_i},
\]
where $P_\alpha$ is a polynomial over $\mathbb{ZP}$ with nonzero constant term. In particular, $x[-\alpha_i]=x_i$.
\end{theorem}

In order to deal with coefficients or frozen variables, we need the following theorem from \cite{FZ07}.  
\begin{theorem}[{\cite[Theorem 3.7]{FZ07}}] \label{FZ07thm3.7}
Let $\mathcal{A}$ be a cluster algebra over an arbitrary semifield $\mathbb{P}$, with a seed at an initial extended cluster $\{x_1,\ldots, x_n,y_1,\ldots, y_n\}$. Then the cluster variables in $\mathcal{A}$ can be expressed as follows:
\[
x_\ell = \frac{ X_\ell |_{\mathcal{F}}(x_1,\ldots, x_n,y_1,\ldots, y_n)}{F_\ell |_{\mathbb{P}}(y_1,\ldots,y_n)},
\]
where $X_\ell$ is the Laurent expression of $x_\ell$ in the case of the principle coefficient, and $F_\ell$ is the specilization of $X_\ell$ evaluating at $x_1=\cdots=x_n=1$, called $F$-polynomial. 
\end{theorem}

\subsection{Category $\mathscr{C}_\xi$ and Hernandez-Leclerc modules}\label{subcategory and Hernandez-Leclerc modules}
Let $I=\{1,2,\ldots,n\}$. Following \cite{BC19}, given a height function $\xi: I \to\mathbb{Z}$ with $|\xi(i)-\xi(i+1)|=1$ for $1\leq i\leq n-1$, we often extend $\xi$ to $[0,n+1]$ by defining $\xi(0)=\xi(2)$ and $\xi(n+1)=\xi(n-1)$, and associate a quiver $Q_\xi$ with the vertex set $I\cup I'$ and arrows defined by 
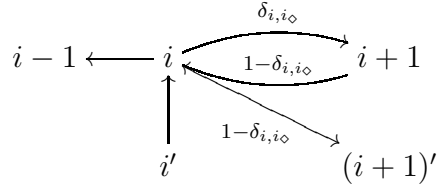
\[
\xymatrix{
i-1 &  i \ar[l] \ar@/^10pt/[rr]^{\delta_{i,i_\diamond}}  \ar[drr]_{1-\delta_{i,i_\diamond}}  &&  i+1 \ar@/^10pt/[ll]_{1-\delta_{i,i_\diamond}}  \\
 & i' \ar[u] &&  (i+1)'}
\]
if $\xi(i)=\xi(i+1)+1$ and otherwise reverse all orientations, where $i_\diamond$ is the minimum integer $\ell\in [i,n]$ such that $\xi(\ell)=\xi(\ell+2)$ if $i<n$ and otherwise $n_\diamond=n$. 

A vertex $i\in [1,n]$ is said to be a \textit{source} or \textit{sink} if $i=n$ or $\xi(i)=\xi(i+2)$, where we ignore all frozen vertices in $Q_\xi$ and require that $1$ is a \textit{source} or \textit{sink} if and only if $\xi(1)=\xi(3)$.

Let $j_\bullet=0$ for $1\leq j \leq 1_\diamond$ and $j_\bullet$ be the maximal source or sink of $Q_\xi$ satisfying $j_\bullet<j$ for $j > 1_\diamond$. For $k\geq 1$ let
\[
\overline{k} = (k+1) (1-\delta_{k,k_\diamond}) + (k_\bullet+1)\delta_{k,k_\diamond}.
\]
\begin{example}\label{example23}
In $A_{9}$, let $\xi(1,2,3,4,5,6,7,8,9)=(-4,-5,-6,-5,-4,-3,-4,-5,-6)$. Then we have the following table.
\begin{table}[H]
\centering
\begin{tabular}{c|c|c|c|c|c|c|c|c|c}
\hline 
$I$ & 1 & 2 & 3 & 4 & 5& 6 & 7& 8 & 9 \\
$\xi$   &  -4 & -5 & -6 & -5 & -4 & -3 & -4 & -5 & -6 \\
$i_\diamond$ & 2 & 2 & 5 & 5 & 5 & 8 & 8 & 8 & 9 \\
$i_\bullet$ & 0 & 0 & 2 & 2 & 2 & 5 & 5 & 5 & 8 \\
$\overline{i}$ & 2 & 1 & 4 & 5 & 3 & 7 & 8 & 6 & 9 \\
\hline 
\end{tabular}
\end{table}
By definition, $Q_\xi$ is the following quiver
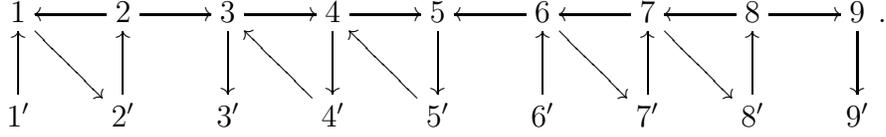
\[
\xymatrix{
1  \ar[dr]  &  2 \ar[l] \ar[r]  & 3 \ar[r]  \ar[d] & 4 \ar[r] \ar[d]  & 5 \ar[d]  & 6 \ar[l] \ar[dr]  & 7 \ar[l]  \ar[dr]  & 8 \ar[l] \ar[r] & 9 \ar[d]  \\
1' \ar[u]  & 2' \ar[u] & 3'  & 4'  \ar[ul]  & 5' \ar[ul]   & 6'  \ar[u]  & 7' \ar[u]  &  8' \ar[u]  & 9'}.
 \]
\end{example}

Let $\mathbb{P}=\text{Trop}(x'_j: j\in I)$ be the tropical semifield generated by $x'_j, j\in I$. Denote by $\mathcal{A}({\bf \widetilde{x}},Q_\xi)$ the \textit{cluster algebra} with an initial seed $({\bf \widetilde{x}},Q_\xi)$, where 
\[
{\bf \widetilde{x}}=\{x_1,\ldots,x_n,x'_1,\ldots,x'_n\},
\] 
$x_1,\ldots,x_n$ are cluster variables and $x'_1,\ldots,x'_n$ are coefficients.

In this paper, we fix $a\in \mathbb{C}^\times$. For simplicity of notation, we write $Y_{i,r}$ for $Y_{i,aq^r}$ for $i\in I$ and $r\in \mathbb{Z}$. Let $\mathcal{P}^+$ be the free abelian monoid generated by variables $Y_{i,\xi(i)\pm 1}$ for $i\in I$. Let $\mathscr{C}_\xi$ be the full subcategory of $\mathscr{C}$ consisting of objects all of whose Jordan-H\"older constituents are indexed by elements of $\mathcal{P}^+$. Simple modules in $\mathscr{C}_\xi$ are of the form $L(m)$, where $m\in \mathcal{P}^+$ and $m$ is called the \textit{highest $\ell$-weight} of $L(m)$. The elements in $\mathcal{P}^+$ are called \textit{dominant monomials}.

Following \cite{BC19}, for $1\leq i < j \leq n$, let $i_2 <\cdots <i_{k-1}$ be an ordered enumeration of the subset 
\[
\{p : i < p < j \mid \xi(p-1)=\xi(p+1) \},
\]
and $i_1=i$, $i_k=j$. Define an element $\omega(i,j)\in \mathcal{P}^+$ by 
\[
\omega(i,j)= Y_{i_1, a_1} Y_{i_2, a_2} \cdots Y_{i_k, a_k},  
\]
where $a_1 = \xi(i)\pm 1$ if $\xi(i+1) = \xi(i) \mp 1$ and $a_\ell = \xi(i_\ell) \pm 1$ if $\xi(i_\ell)=\xi(i_\ell-1)\pm 1$ for $\ell \geq 2$.

Let $\mathbf{Pr}_{\xi}=\{ Y_{i,\xi(i)\pm1} \mid i\in I \}\cup \{\omega(i,j) \mid i, j\in I, i\neq j\}$ and $\mathbf{f}=\{ f_i= Y_{i,\xi(i)-1} Y_{i,\xi(i)+1} \mid i\in I\}$.
A simple module $L(m)$ for $m\in \mathbf{Pr}_{\xi}\cup \mathbf{f}$ is called an \textit{Hernandez-Leclerc module} by Brito and Chari \cite{BC19}, which are precisely all the prime objects in this category. The simple modules $L(f_i)$, $i\in I$, are Kirillov-Reshitikhin modules and $L(Y_{i,r})$, $i\in I$, are fundamental modules. We are interested in $L(\omega(i,j))$ for $1\leq i < j \leq n$. 

Following Fomin and Zelevinsky's result \cite{FZ03}, the set of cluster variables in $\mathcal{A}({\bf \widetilde{x}},Q_\xi)$ is in bijection with the set $\Phi_{\geq -1}$ of almost positive roots in the root system of type $A_n$. Let $\{\alpha_i \mid i\in I\}$ be a set of simple roots of type $A_n$ and let $\alpha_{i,j}=\alpha_i + \cdots + \alpha_j$ for $1\leq i\leq j\leq n$. Denote all cluster variables and coefficients by 
\[
\{x_i:=x[-\alpha_i],  x[\alpha_{i,j}], x'_i  \mid 1\leq i\leq j\leq n\}. 
\]
Following \cite{BC19}, the cluster variable $x[\alpha_{i,j}]$ is obtained by mutating the sequence $i,i+1,\ldots,j$ at the initial cluster $\{x_i \mid 1\leq i \leq n\}$.

Let $\mathcal{K}_0(\xi)$ be the Grothendieck ring of $\mathscr{C}_\xi$. Denote by $[M]\in \mathcal{K}_0(\xi)$ the equivalence class of $M\in \mathscr{C}_\xi$. Brito and Chari proved that $\mathscr{C}_\xi$ is a monoidal category of the cluster algebra $\mathcal{A}({\bf \widetilde{x}},Q_\xi)$.

\begin{theorem}\cite[Theorem 1, Corollary in Section 1.3]{BC19} \label{BCTheorem1} 
Let $\xi: I \to \mathbb{Z}$ be a height function.  Then there is an isomorphism of rings $\iota: \mathcal{A}({\bf \widetilde{x}},Q_\xi) \to \mathcal{K}_0(\xi)$ such that 
\begin{align*}
& \iota(x[-\alpha_i]) = [L(Y_{i,\xi(i+1)})], \quad \iota(x'_i) = [L(Y_{i,\xi(i)-1}Y_{i,\xi(i)+1})], \\
& \iota(x[\alpha_{i,i_\diamond}]) = [L(Y_{i,\xi(i+1)\pm2})], \quad \xi(i) = \xi(i+1) \pm  1, \\
& \iota(x[\alpha_{i,k}]) = [\omega(i,\overline{k})], \quad k\neq i_\diamond, \\
& \iota(x'_i x[\alpha]) = [x'_i \omega], \quad i \in I, \ \alpha\in \Phi_{\geq -1}, \omega = \iota(x[\alpha]).
\end{align*}
Moreover, $\iota$ sends a cluster variable (respectively, cluster monomial) to a real prime simple object (respectively, real simple object)  of $\mathscr{C}_\xi$. In particular, $\mathscr{C}_\xi$ is a monoidal categorication of $ \mathcal{A}({\bf \widetilde{x}},Q_\xi)$. 
\end{theorem}
The proof given by Brito and Chari in \cite{BC19} used the fact that the set of cluster monomials forms a linear basis of cluster algebras of finite type \cite{CKLP13,D11}.

\subsection{Snake graphs}\label{snake graphs}

Fix an orthonormal basis of the plane $\mathbb{R}^2$. A tile $G$ is a square of fixed side-length with four vertices and four edges in the plane whose sides are parallel or orthogonal to the chosen basis. Following \cite{CS13}, a \textit{snake graph} is a connected graph consisting of finitely many tiles $G_1, G_2, \ldots, G_d$ with $d\geq 1$, such that for each $i=1,\ldots,d-1$
\begin{itemize}
\item[(i)] $G_i$ and $G_{i+1}$ share exactly an edge $e_i$ and the edge is either the north edge of $G_i$ and the south edge of $G_{i+1}$ or the east edge of $G_i$ and the west edge of $G_{i+1}$.
\item[(ii)]  $G_i$ and $G_j$ have no edge in common whenever $|i-j|\geq 2$.
\item[(iii)]  $G_i$ and $G_j$ are disjoint whenever $|i-j|\geq 3$.
\end{itemize}

Let $\mathcal{G}=(G_1,G_2, \ldots, G_d)$ be a snake graph with tiles $G_1, G_2, \ldots, G_d$. The $d-1$ edges $e_1, e_2, \ldots, e_{d-1}$ are called \textit{interior edges} of $\mathcal{G}$ and the rest of edges are called \textit{boundary edges}. Denote by $_S\mathcal{G}$  (respectively, $_W\mathcal{G}$)  the south (respectively, west) edge of the first tile of $\mathcal{G}$ and denote by $\mathcal{G}^N$  (respectively, $\mathcal{G}^E$) the north (respectively, east) edge of the last tile of $\mathcal{G}$.  

A snake graph $\mathcal{G}$ is called \textit{straight} if all its tiles lie in one column or one row, and a snake graph is called \textit{zigzag} if no three consecutive tiles are straight \cite{CS18,CS19}.

A \textit{perfect matching} $P$ of a snake graph $\mathcal{G}$ is a subset of the set of edges of $\mathcal{G}$ such that each vertex of $\mathcal{G}$ is exactly in one edge in $P$. Denote by $\text{Match}(\mathcal{G})$ the set of all perfect matchings of $\mathcal{G}$. A snake graph $\mathcal{G}$ has precisely two perfect matchings, called the \textit{minimal matching} $P_{-}$ and the \textit{maximal matching} $P_{+}$ of $\mathcal{G}$, which contain only boundary edges. 

A snake graph is called a \textit{labeled snake graph} if each edge and each tile in the snake graph carries a label or weight \cite{CS15}. For our snake graphs, these labels are cluster variables. Without confusion, we still use $\mathcal{G}$ to denote a labeled snake graph. Assume that the edges of a perfect matching $P$ of $\mathcal{G}$ are labeled by $v_1, v_2, \ldots, v_r$.  Following \cite{MSW11}, one defines the \textit{weight monomial} $x(P)$ of $P$ by $x(P)=\prod_{i=1}^r x_{v_i}$. 

Given a snake graph $\mathcal{G}$, let $e_0=_S\mathcal{G}$ and choose an edge $e_d \in \{\mathcal{G}^N, \mathcal{G}^E\}$. In \cite{CS18}, Canakci and Schiffler defined a sign function   
\[
f: \{e_0, e_1, e_2,\ldots, e_{d-1}, e_d\} \to \{\pm \},
\]
such that on every tile in $\mathcal{G}$ the north edge and the west edge have the same sign, the south edge and the east edge have the same sign, and the sign on the north edge is opposite to the sign on the south edge.

Let $a_i \in \mathbb{Z}_{\geq 1}$ be the number of the sign in a maximal subsequence of constant sign appearing in $(f(e_0), f(e_1), f(e_2),\ldots, f(e_{d-1}), f(e_d))$. As an illustrated example, the sign function of the following snake graph is $(-,-,+,+,+,-,-,-,-,-)$, and $a_1=2,a_2=3,a_3=5$.  
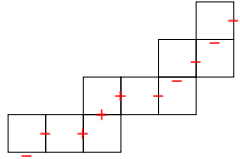
\begin{figure}[H]
\resizebox{.5\width}{.5\height}{
\begin{tikzpicture}
\draw (0,0) rectangle (1,1);
\draw (1,0) rectangle (2,1);
\draw (2,0) rectangle(3,1);
\draw (2,1) rectangle(3,2);
\draw (3,1) rectangle(4,2);
\draw (4,1) rectangle(5,2);
\draw (4,2) rectangle(5,3);
\draw (5,2) rectangle(6,3);
\draw (5,3) rectangle(6,4);
\node[red] at (0.5,-0.1) {$\pmb{-}$};
\node[red] at (1,0.5) {$\pmb{-}$};
\node[red] at (2,0.5) {$\pmb{+}$};
\node[red] at (2.5,1) {$\pmb{+}$};
\node[red] at (3,1.5) {$\pmb{+}$};
\node[red] at (4,1.5) {$\pmb{-}$};
\node[red] at (4.5,1.9) {$\pmb{-}$};
\node[red] at (5,2.4) {$\pmb{-}$};
\node[red] at (5.5,2.9) {$\pmb{-}$};
\node[red] at (6,3.5) {$\pmb{-}$};
\end{tikzpicture}}
\end{figure}

For a snake graph $\mathcal{G}$ with a sign function $(a_1,a_2,\ldots,a_n)$, let $\ell_i=\sum_{j=1}^i a_j$ for $1\leq i \leq n$ and we agree that $\ell_0=0$. Following \cite{CS18}, one can define zigzag subsnake graphs $\mathcal{H}_1,\ldots,\mathcal{H}_n$ of  $\mathcal{G}$ as follows. Let 

\begin{align*}
& \mathcal{H}_1 =(G_1, G_2, \ldots, G_{\ell_1-1}),  \\
& \qquad \quad \vdots \qquad \qquad  \vdots \\
& \mathcal{H}_i =(G_{\ell_{i-1}+1}, \ldots, G_{\ell_i-1}), \\
& \qquad \quad \vdots \qquad \qquad  \vdots \\
& \mathcal{H}_n =(G_{\ell_{n-1}+1}, \ldots, G_d),
\end{align*}
where $(G_j,\ldots,G_k)$ is the zigzag subsnake graph with the tiles $G_j,G_{j+1},\ldots,G_k$ if $j\leq k$, and $(G_{j+1},\ldots,G_j)$ is the single edge $e_j$. The decomposition is in fact obtained by deleting the sign-changed tiles. 

Following \cite{Rab18}, once we choose the minimal perfect matching $P_{-}$ of a snake graph $\mathcal{G}$, then the minimal matching $P_{-}|_{\mathcal{H}_i}$ of a subsnake graph $\mathcal{H}_i$ is either the matching which inherits from $P_{-}$ or the union of the matching which inherits from $P_{-}$ and a unique interior edge. 

For an arbitrary perfect matching $P$ of  $\mathcal{G}$, the symmetric different $P_{-} \ominus P$ is defined as
\[
P_{-} \ominus P=(P_{-} \cup P) \backslash (P_{-} \cap P).
\] 
\begin{lemma} \cite[Lemma 4.8]{MSW11}
The set $P_{-} \ominus P$ is the set of boundary edges of a (possibly disconnected) subgraph $G_P$ of $\mathcal{G}$, which is a union of cycles. These cycles enclose a set of ties $\cup_{j\in J} G_{i_j}$, where $J$ is a finite index set.
\end{lemma}

From now on, assume that the label of a tile $G_i$ is $\tau_i$ in a labeled snake graph $\mathcal{G}$. Following \cite[Definition 4.9]{MSW11}, with the notation above,  the \textit{height monomial} $y(P)$ of a perfect matching $P$ of $\mathcal{G}$ is defined by 
\[
y(P) = \prod_{j\in J} y_{\tau_{i_j}}.
\]
 
\subsection{Continued fractions and $F$-polynomials}
Following \cite{CS18,Rab18}, a finite continued fraction 
\begin{align*}
[a_1,a_2,\ldots,a_n] : = a_1 + \cfrac{1}{ a_2 + \cfrac{1}{ \ddots + \cfrac{1}{a_n}}} 
\end{align*}
is said to be \textit{positive} if each $a_i$ is a positive integer. Denote by $\mathcal{N}[a_1,a_2,\ldots,a_n]$ the numerator of the continued fraction. Then $\mathcal{N}[a_1,a_2,\ldots,a_n]$ is computed by the recursion
\[
\mathcal{N}[a_1,a_2,\ldots,a_n] = a_n \mathcal{N}[a_1,a_2,\ldots,a_{n-1}] + \mathcal{N}[a_1,a_2,\ldots,a_{n-2}],
\]
where $\mathcal{N}[a_1]=a_1$ and $\mathcal{N}[a_1,a_2]=a_1a_2+1$. We refer the reader to \cite[Proposition 2.1]{RS20} for more properties of continued fractions.

For a positive continued fraction $[a_1,a_2,\ldots,a_n]$, let $\ell_i=\sum_{s=1}^i a_s$ and we agree $\ell_0=0$. In \cite{CS18}, Canakci and Schiffler established a bijection between snake graphs and positive continued fractions via the sign function of a snake graph. Fix a positive continued fraction $[a_1,a_2,\ldots,a_n]$, the corresponding snake graph $\mathcal{G}[a_1,a_2,\ldots,a_n]$ consists of tiles $G_1, G_2, \ldots, G_{\ell_n-1}$ and has the sign function of the following form
\[
(\underbrace{-,\ldots,-}_{a_1},  \underbrace{+,\ldots,+}_{a_2}, \underbrace{-,\ldots,-}_{a_3},  \ldots, \underbrace{\pm,\ldots,\pm}_{a_n}).
\]
The snake graph $\mathcal{G}[a_1,a_2,\ldots,a_n]$ has $\mathcal{N}[a_1,a_2,\ldots,a_n]$ many perfect matchings \cite[Theorem 3.4]{CS18}.

For simplicity, let $\prod_{\ell=i}^j x_\ell=1$ for $j<i$. The following results from \cite{Rab18} will play an important role in the subsequent proof. In \cite{Rab18} Rabideau required that $_S\mathcal{G}\in P_{-}$ for a labeled snake graph $\mathcal{G}$. 
 
\begin{lemma}\cite[Lemma 3.3]{Rab18} \label{Rab3.3}
The $F$-polynomial associated to a zigzag snake graph $(G_1,G_2,\ldots,G_d)$ (the label of $G_j$ is $\tau_j$) is 
\[
\sum_{k=0}^d \prod_{j=1}^k y_{\tau_j}  \quad \text{ or } \quad  \sum_{k=1}^{d+1} \prod_{j=k}^d y_{\tau_j},
\] 
depending on whether the minimal perfect matching $P_{-}$ contains a pair of opposite boundary edges of $G_1$.
\end{lemma}

In \cite[Definition 3.1]{Rab18},  Rabideau defined a continued fraction of Laurent polynomial $[\mathcal{L}_1,\mathcal{L}_2,\ldots,\mathcal{L}_n]$, where each $\mathcal{L}_i=\varphi_iC_i$ and 
\begin{align*}
C_i=\begin{cases}
\prod\limits_{j=1}^{\ell_{i-1}} y_{\tau_j} & \text{if $i$ is odd}, \\
\prod\limits_{j=1}^{\ell_{i}-1} y^{-1}_{\tau_j} & \text{if $i$ is even}, \\
\end{cases} \quad 
\varphi_i = \begin{cases}
\sum\limits_{k=\ell_{i-1}}^{\ell_i-1} \prod\limits_{j=\ell_{i-1}+1}^{k} y_{\tau_j} & \text{if $i$ is odd}, \\
\sum\limits_{k=\ell_{i-1}+1}^{\ell_i} \prod\limits_{j=k}^{\ell_i-1} y_{\tau_j} & \text{if $i$ is even}. \\
\end{cases} 
\end{align*} 
Here $\varphi_i$ is in fact the $F$-polynomial associated to a certain zigzag snake graph. Note that we revise the subscripts of the letter $y$, because we agree that the label of $G_j$ is $\tau_j$.

\begin{theorem}\cite[Theorem 3.4]{Rab18} \label{Rab3.4}
The $F$-polynomial associated to the snake graph of the positive continued fraction $[a_1,a_2,\ldots,a_n]$, $a_1>1$, is given by the following equation 
\[
F(\mathcal{G}[a_1,a_2,\ldots,a_n]) = \begin{cases}
\mathcal{N}[\mathcal{L}_1,\mathcal{L}_2,\ldots,\mathcal{L}_n] & \text{if $n$ is odd}, \\
C^{-1}_n\mathcal{N}[\mathcal{L}_1,\mathcal{L}_2,\ldots,\mathcal{L}_n] & \text{if $n$ is even}, \\
\end{cases}
\] 
where $\mathcal{N}[\mathcal{L}_1,\mathcal{L}_2,\ldots,\mathcal{L}_n]$ is defined by the recursion 
\[
\mathcal{N}[\mathcal{L}_1,\mathcal{L}_2,\ldots,\mathcal{L}_n]=\mathcal{L}_n\mathcal{N}[\mathcal{L}_1,\mathcal{L}_2,\ldots,\mathcal{L}_{n-1}]+\mathcal{N}[\mathcal{L}_1,\mathcal{L}_2,\ldots,\mathcal{L}_{n-2}]
\]
where $\mathcal{N}[\mathcal{L}_1]=\mathcal{L}_1$ and $\mathcal{N}[\mathcal{L}_1,\mathcal{L}_2]=\mathcal{L}_1\mathcal{L}_2+1$.
\end{theorem}

\section{$q$-characters of Hernandez-Leclerc modules and snake graphs}\label{The definition of HL-modules}
In this section, we give a formula for $q$-characters of Hernandez-Leclerc modules using snake graphs, which involves an explicit formula for $F$-polynomials. 

\subsection{A combinatorial description of Hernandez-Leclerc modules} 
We recall the definition of Hernandez-Leclerc modules in \cite{BC19}. Let $I=\{1,2,\ldots,n\}$. Recall that $\mathscr{C}_\xi$ is the full subcategory of $\mathscr{C}$ consisting of objects all of whose Jordan-H\"older constituents are indexed by elements of $\mathcal{P}^+$, see Section \ref{subcategory and Hernandez-Leclerc modules}. Brito and Chari proved that $\mathscr{C}_\xi$ is a monoidal categorification of the cluster algebra $\mathcal{A}({\bf \widetilde{x}},Q_\xi)$. They call these modules in $\mathscr{C}_\xi$, which correspond to cluster variables and frozen variables, Hernandez-Leclerc modules. 


From now on, we call an Hernandez-Leclerc module an HL-module for simplicity.

\begin{theorem}\label{HL defintion equivalent}
An HL-module corresponding to a cluster variable is a simple $U_{q}(\widehat{\mathfrak{g}})$-module with the highest $\ell$-weight monomial
\begin{align} \label{equation: HL}
Y_{i_1,a_1}Y_{i_2,a_2}\cdots Y_{i_k,a_k},
\end{align}
where $k\in\mathbb{Z}_{\geq 1}$, $i_j\in I$, $a_j \in \mathbb{Z}$ for $j=1,2,\ldots,k$, and
\begin{itemize}
\item[{\rm (i)}] $i_1<i_2<\cdots<i_k$,
\item[{\rm (ii)}] $(a_{j}-a_{j-1})(a_{j+1}-a_{j})<0$ for $2\leq j\leq k-1$,
\item[{\rm (iii)}] $|a_{j}-a_{j-1}|=i_{j}-i_{j-1}+2$ for $2\leq j\leq k$.
\end{itemize}
\end{theorem}

\begin{proof}
In \cite{BC19}, the definition of HL-modules depends on the choice of a map $\xi: I \to \mathbb{Z}$. We will prove that every simple $U_{q}(\widehat{\mathfrak{g}})$-module with the highest $\ell$-weight monomial (\ref{equation: HL}) determines a map $\xi: I \to \mathbb{Z}$ and hence it is an HL-module in the sense of Brito and Chari. Indeed, let $m=Y_{i_1,a_1}Y_{i_2,a_2}\cdots Y_{i_k,a_k}$ for $i_1<i_2<\cdots<i_k$. Let $i=i_1$, $j=i_k$. We split the proof into the following two cases.

{\bf Case 1}.  If $a_{1}<a_{2}$, we define $\xi(i)=a_{1}+1$, $\xi(i+1)=a_{1}+2,\ldots,\xi(i_2)=a_{1}+i_2-i_1+1$. Condition~(iii) implies that $\xi(i_2)=a_{1}+i_2-i_1+1=a_{2}-1$. Using Condition~(ii) and the same argument as before, we define a map $\xi: [i,j] \to \mathbb{Z}$ by 
\[ 
\xi(x)=
\begin{cases}
a_{2\ell-1}+x-i_{2\ell-1}+1, & \text{if $x \in [i_{2\ell-1},i_{2\ell}]$;} \\
a_{2\ell}-x+i_{2\ell}-1, & \text{if $x \in[i_{2\ell},i_{2\ell+1}]$.}
\end{cases}
\]
We extend $\xi$ to the domain $[1,n]$ subject to 
\begin{align*}
\begin{cases}
\xi(x)-\xi(x+1)=1 & \text{if $1\leq x \leq i-1$}, \\
\xi(x+1)=\xi(x-1) &  \text{if $j\leq x \leq n-1$}.  
\end{cases}
\end{align*}
So the vertices $(j-1), j, \ldots, n$ are sources or sinks. It follows from Theorem \ref{BCTheorem1} that 
\begin{align*}
\iota(x[\alpha_{i,j}]) = \begin{cases} 
 [L(Y_{i,a_1})],  & \text{if $j=i_\diamond$,} \\
 [\omega(i,j)] = [L(Y_{i_1,a_1}Y_{i_2,a_2}\ldots Y_{i_k,a_k})] & \text{if $j \neq i_\diamond$,} \\
\end{cases}
\end{align*}
which is an HL-module in the sense of Brito and Chari.

{\bf Case 2}.  If $a_{1}>a_{2}$, we define $\xi(i)=a_{1}-1$, $\xi(i+1)=a_{1}-2,\ldots,\xi(i_2)=a_{1}-i_2+i_1-1$. By Condition~(iii), we have $\xi(i_2)=a_{1}-i_2+i_1-1=a_{2}+1$. Using Condition~(ii) and the same argument as before, we define a map $\xi: [i,j] \to \mathbb{Z}$ by 
\begin{align*}
\xi(x)=\begin{cases}
a_{2\ell-1} -x + i_{2\ell-1}-1, & \hbox{$i\in [i_{2\ell-1},i_{2\ell}]$;} \\
a_{2\ell} + x - i_{2\ell}+1, & \hbox{$i\in[i_{2\ell},i_{2\ell+1}]$.}
\end{cases}
\end{align*}
We extend $\xi$ to the domain $[1,n]$ subject to 
\begin{align*}
\begin{cases}
\xi(x+1)-\xi(x)=1 & \text{if $1\leq x \leq i-1$}, \\
\xi(x+1)=\xi(x-1) &  \text{if $j\leq x \leq n-1$}.  
\end{cases}
\end{align*}
So the vertices $(j-1), j, \ldots, n$ are sources or sinks. It follows from Theorem \ref{BCTheorem1} that 
\begin{align*}
\iota(x[\alpha_{i,j}]) = \begin{cases} 
 [L(Y_{i,a_1})],  & \text{if $j=i_\diamond$,} \\
 [\omega(i,j)] = [L(Y_{i_1,a_1}Y_{i_2,a_2}\ldots Y_{i_k,a_k})] & \text{if $j \neq i_\diamond$,} \\
\end{cases}
\end{align*}
which is an HL-module in the sense of Brito and Chari.

Conversely, for $\omega_{i,\xi_{i}\pm1}\in \mathbf{Pr}_{\xi}$, obviously, it can be written into the form (\ref{equation: HL}). For any $\omega_{i,j}=\omega_{i_{1},a_{1}}\ldots\omega_{i_{k},a_{k}}\in \mathbf{Pr}_{\xi}$, by the definition of $\omega_{i,j}$, we have
\begin{equation*}
i_{1}<i_{2}<\ldots<i_{k},
\end{equation*}
and the function $\xi$ must be a strictly increasing height function or a strictly decreasing height function on these intervals $[i_{1},i_{2}],\ldots, [i_{k-1},i_{k}]$, and the strictly increasing intervals and the strictly decreasing intervals appear alternatively.

Suppose that $\xi$ is a strictly increasing (respectively, decreasing) function on the interval $[i_{\ell},i_{\ell+1}]$ for a certain $1\leq \ell\leq k-1$. Then 
\begin{align*}
& a_{\ell}=\xi(i_{\ell})-1, \  a_{\ell+1} =\xi(i_{\ell+1})+1  \\
& (\text{respectively, } a_{\ell}=\xi(i_{\ell})+1, a_{\ell+1}=\xi(i_{\ell+1})-1). 
\end{align*}
Since $\xi$ is a strictly decreasing (respectively, increasing) function on the interval $[i_{\ell-1},i_\ell]$, we have
\begin{align*}
(a_{\ell}-a_{\ell-1})(a_{\ell+1}-a_{\ell})=(\xi(i_{\ell})-\xi(i_{\ell-1})-2)(\xi(i_{\ell+1})-\xi(i_{\ell})+2)<0,
\end{align*}
and 
\begin{align*}
& |a_{\ell+1}-a_{\ell}|=|\xi(i_{\ell+1})-\xi(i_{\ell})+2|=|(i_{\ell+1}-i_{\ell})+2|=i_{\ell+1}-i_{\ell}+2, \\
& |a_\ell-a_{\ell-1}|=|\xi(i_{\ell})-\xi(i_{\ell-1})-2|=|\xi(i_{\ell-1})+2-\xi(i_{\ell})|=i_\ell-i_{\ell-1}+2.
\end{align*}
Similarly, it holds for the case that $\xi$ is a strictly decreasing (respectively, increasing) function on the interval $[i_{\ell},i_{\ell+1}]$. 

Therefore, every HL-module has its highest $\ell$-weight monomial of the form (\ref{equation: HL}).
\end{proof}

\begin{remark}
The same HL-module may correspond to different height functions $\xi$, even in the same type. 
\end{remark}

\subsection{Hernandez-Leclerc modules and snake graphs}\label{HL-modules and snake graphs}
In a cluster algebra of finite type from a surface without punctures, arcs in an initial triangulation of the surface correspond to initial variables which are parameterized by negative simple roots \cite[Theorem 1.9]{FZ03}. The arc crossing initial arcs $-\alpha_i,-\alpha_{i+1},\ldots,-\alpha_j$ corresponds to the cluster variable parameterized by $\alpha_{i,j}$, where $i\leq j$.
Fix a height function $\xi$, we shall construct a  unique labeled snake graph for a non-initial HL-module.  In general we have the relationships shown in Figure \ref{four correspondences}.

From \cite{FST08}, also see \cite[Chapter 3]{S14}, it follows that up to rotation a triangulation of a surface determines a quiver, and the quiver completely reflects the configuration of arcs in the triangulation.

Let $\xi$ be a height function. Denote by $Q$ a connected full subquiver of $Q_\xi$, where we ignore the frozen vertices. Assume without loss of generality that the set of vertices in $Q$ is $\{i,i+1,\ldots,j\}$ from left to right in order.  We define the snake graph 
\[
\mathcal{G}=(G_i,G_{i+1},\ldots,G_j)
\]
associated to $Q$ as follows. The first two tiles are placed in the same horizontal line, that is, the east edge of $G_i$ and the west edge of $G_{i+1}$ are the same. For each $i+2 \leq \ell \leq j$, the tile $G_\ell$ is placed such that $(G_{\ell-2},G_{\ell-1},G_\ell)$ is straight if the $(\ell-1)$-th vertex is a source or sink, and otherwise $(G_{\ell-2},G_{\ell-1},G_\ell)$ is zigzag.  From now on, we use a digit ``1" to denote the cluster variable $x_1$, and so on. Label each tile $G_\ell$ by $\ell$ in the interior of the tile. For any two consecutive tiles $(G_\ell,G_{\ell+1})$, each edge is labeled near the edge obeying the following rules.
\begin{itemize}
\item[(i)] If $G_\ell$ and $G_{\ell+1}$ share exactly an edge $e_\ell$ and the edge is the east edge of $G_\ell$ and the west edge of $G_{\ell+1}$, then the north edge of $G_\ell$ is labeled by $(\ell+1)$ and the south edge of $G_{\ell+1}$ is labeled by $\ell$.
\item[(ii)] If $G_\ell$ and $G_{\ell+1}$ share exactly an edge $e_\ell$ and the edge is the north edge of $G_\ell$ and the south edge of $G_{\ell+1}$, then the east edge of $G_\ell$ is labeled by $(\ell+1)$ and the west edge of $G_{\ell+1}$ is labeled by $\ell$.  
\end{itemize}
We leave edges $_W\mathcal{G}$, $_S\mathcal{G}$, $\mathcal{G}^N$, and $\mathcal{G}^E$ no label.  In particular, every edge has no label for a snake graph consisting of a tile. 

The minimal (respectively, maximal) matching $P_{-}$ (respectively, $P_{+}$) of $\mathcal{G}$ is chosen as follows. If there is an arrow $(i+1) \to i$ in $Q$, $P_{-}$ is defined as the unique matching which contains only boundary edges and also contains $_W\mathcal{G}$, $P_{+}$ is the other matching with only boundary edges. Otherwise $P_{-}$ is defined as the unique matching which contains only boundary edges and contains $_S\mathcal{G}$, $P_{+}$ is the other matching with only boundary edges. If $Q$ consists of a single vertex, then we agree that $_S\mathcal{G}\in P_{-}$, $_W\mathcal{G}\in P_+$. 

\begin{definition}\label{HL and snake graph}
The labeled snake graph associated to the HL-module parameterized by $\alpha_{i,j}$ is the labeled snake graph determined by the connected full subquiver with vertex set $\{i,i+1,\ldots,j\}$ in $Q_\xi$. 
\end{definition}

The following example explains a snake graph associated to an HL-module.
\begin{example}\label{example34}
Continue our previous Example \ref{example23} and consider the cluster variable $x[\alpha_{1,7}]$ and the corresponding HL-module $L(Y_{1,-3}Y_{3,-7}Y_{6,-2}Y_{8,-6})$. The full subquiver of $Q_\xi$ with vertex set $\{1,2,3,4,5,6,7\}$ is the following quiver 

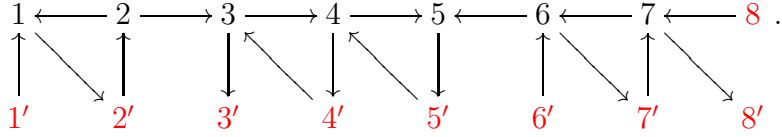
\[
\xymatrix{
1  \ar[dr]  &  2 \ar[l] \ar[r]  & 3 \ar[r]  \ar[d] & 4 \ar[r] \ar[d]  & 5 \ar[d]  & 6 \ar[l] \ar[dr]  & 7 \ar[l]  \ar[dr]  & \textcolor{red}{8} \ar[l]  \\
\textcolor{red}{1'} \ar[u]  & \textcolor{red}{2'} \ar[u] & \textcolor{red}{3'}  & \textcolor{red}{4'}  \ar[ul]  & \textcolor{red}{5'} \ar[ul]   & \textcolor{red}{6'}  \ar[u]  & \textcolor{red}{7'} \ar[u]  &  \textcolor{red}{8'}}.
\]

Following Definition \ref{HL and snake graph}, the labeled snake graph associated to $L(Y_{1,-3}Y_{3,-7}Y_{6,-2}Y_{8,-6})$ is shown as follows.
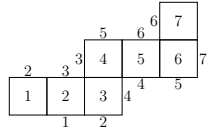
\begin{figure}[H]
\resizebox{.5\width}{.5\height}{
\begin{tikzpicture}
\draw (0,0) rectangle (1,1);
\draw (1,0) rectangle (2,1);
\draw (2,0) rectangle(3,1);
\draw (2,1) rectangle(3,2);
\draw (3,1) rectangle(4,2);
\draw (4,1) rectangle(5,2);
\draw (4,2) rectangle(5,3);
\node at (0.5,0.5) {1};
\node at (1.5,0.5) {2};
\node at (2.5,0.5) {3};
\node at (2.5,1.5) {4};
\node at (3.5,1.5) {5};
\node at (4.5,1.5) {6};
\node at (4.5,2.5) {7};
\node[above] at (0.5,0.9) {2};
\node[above] at (1.5,0.9) {3};
\node[below] at (1.5,0.1) {1};
\node[below] at (2.5,0.1) {2};
\node[left] at (2.1,1.5) {3};
\node[right] at (2.9,0.5) {4};
\node[below] at (3.5,1.1) {4};
\node[above] at (2.5,1.9) {5};
\node[below] at (4.5,1.1) {5};
\node[above] at (3.5,1.9) {6};
\node[left] at (4.1,2.5) {6};
\node[right] at (4.9,1.5) {7};
\end{tikzpicture}}
\end{figure}
\end{example}

The same HL-module may correspond to different labeled snake graphs, even in the same type. Consider a height function 
\[
\xi(1,2,3,4,5,6,7,8,9)=(-4,-5,-6,-5,-4,-3,-4,-5,-4)
\] 
in $A_{9}$. The quiver $Q_\xi$ is the following quiver.
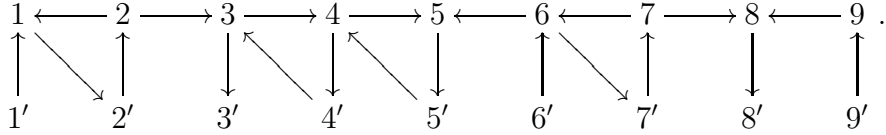
\[
\xymatrix{
1  \ar[dr]  &  2 \ar[l] \ar[r]  & 3 \ar[r]  \ar[d] & 4 \ar[r] \ar[d]  & 5 \ar[d]  & 6 \ar[l] \ar[dr]  & 7 \ar[l]  \ar[r]  & 8 \ar[d] & 9 \ar[l]  \\
1' \ar[u]  & 2' \ar[u] & 3'  & 4'  \ar[ul]  & 5' \ar[ul]   & 6'  \ar[u]  & 7' \ar[u]  &  8'  & 9' \ar[u] }.
 \]
By Theorem \ref{BCTheorem1}, we have $\iota(x[\alpha_{1,8}])= [L(Y_{1,-3}Y_{3,-7}Y_{6,-2}Y_{8,-6})]$. Using Definition \ref{HL and snake graph}, the labeled snake graph associated to $L(Y_{1,-3}Y_{3,-7}Y_{6,-2}Y_{8,-6})$ is shown as follows.
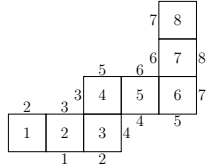
\begin{figure}[H]
\resizebox{.5\width}{.5\height}{
\begin{minipage}{0.5\linewidth}
\begin{tikzpicture}
\draw (0,0) rectangle (1,1);
\draw (1,0) rectangle (2,1);
\draw (2,0) rectangle(3,1);
\draw (2,1) rectangle(3,2);
\draw (3,1) rectangle(4,2);
\draw (4,1) rectangle(5,2);
\draw (4,2) rectangle(5,3);
\draw (4,3) rectangle(5,4);
\node at (0.5,0.5) {1};
\node at (1.5,0.5) {2};
\node at (2.5,0.5) {3};
\node at (2.5,1.5) {4};
\node at (3.5,1.5) {5};
\node at (4.5,1.5) {6};
\node at (4.5,2.5) {7};
\node at (4.5,3.5) {8};
\node[above] at (0.5,0.9) {2};
\node[above] at (1.5,0.9) {3};
\node[below] at (1.5,0.1) {1};
\node[below] at (2.5,0.1) {2};
\node[left] at (2.1,1.5) {3};
\node[right] at (2.9,0.5) {4};
\node[below] at (3.5,1.1) {4};
\node[above] at (2.5,1.9) {5};
\node[below] at (4.5,1.1) {5};
\node[above] at (3.5,1.9) {6};
\node[left] at (4.1,2.5) {6};
\node[right] at (4.9,1.5) {7};
\node[left] at (4.1,3.5) {7};
\node[right] at (4.9,2.5) {8};
\end{tikzpicture}
\end{minipage}}
\end{figure}

\begin{lemma}
If a height function $\xi$ is fixed, then an HL-module corresponding to a non-initial cluster variable (excluding frozen variables) determines a unique labeled snake graph. 
\end{lemma}
\begin{proof}
Since a height function $\xi$ is fixed, by Theorem \ref{BCTheorem1}, we have a cluster algebra isomorphism  $\iota: \mathcal{A}({\bf \widetilde{x}},Q_\xi) \to \mathcal{K}_0(\xi)$ and HL-modules corresponding to non-initial cluster variables are in bijection with almost positive roots. The result follows from Definition \ref{HL and snake graph}.
\end{proof}

\subsection{$q$-characters of Hernandez-Leclerc modules in term of perfect matchings of snake graphs}

To our purpose, we first give the dual version of Theorem \ref{Rab3.4} by the following theorem whose proof is completely similar to one given by Rabideau in \cite[Theorem 3.4]{Rab18} except using the following formula from \cite{CS13,CS15}.
\begin{align}\label{dual formula}
F(\mathcal{G}[a_1,a_2,\ldots,a_n])= y_{34} F(\mathcal{G}[a_1,a_2,\ldots,a_{n-1}]) F(\mathcal{H}_n) + y_{56}F(\mathcal{G}[a_1,a_2,\ldots,a_{n-2}]),  
\end{align}
where the variables $y_{34}$ and $y_{56}$ are defined as follows, $y_0=1$.
\begin{align*}
y_{34} =\begin{cases}
1  & \text{if $n$ is odd}, \\
y_{\tau_{\ell_{n-1}}}  & \text{if $n$ is even},
\end{cases} \quad 
y_{56} =\begin{cases}
\prod_{j=\ell_{n-2}}^{\ell_n-1} y_{\tau_j} & \text{if $n$ is odd}, \\
1  & \text{if $n$ is even}.
\end{cases}
\end{align*}

\begin{theorem}\label{dual case of F-polynomial}
Suppose that in $\mathcal{G}[a_1,a_2,\ldots,a_n]$ $(a_1>1)$ the minimal perfect matching contains the west edge of the first tile. Then the  $F$-polynomial associated to $\mathcal{G}[a_1,a_2,\ldots,a_n]$ is given by the following equation
\[
F(\mathcal{G}[a_1,a_2,\ldots,a_n]) = \begin{cases}
C^{-1}_n\mathcal{N}[\mathcal{L}_1,\mathcal{L}_2,\ldots,\mathcal{L}_n] & \text{if $n$ is odd}, \\
\mathcal{N}[\mathcal{L}_1,\mathcal{L}_2,\ldots,\mathcal{L}_n]  & \text{if $n$ is even},
\end{cases}
\]
where each $\mathcal{L}_i=\varphi_iC_i$ and 
\begin{align*}
C_i=\begin{cases}
\prod\limits_{j=\ell_1}^{\ell_{i}-1} y^{-1}_{\tau_j}  & \text{if $i$ is odd}, \\
\prod\limits_{j=\ell_1}^{\ell_{i-1}} y_{\tau_j}  & \text{if $i$ is even}, \\
\end{cases} \quad 
\varphi_i = \begin{cases}
\sum\limits_{k=\ell_{i-1}+1}^{\ell_i} \prod\limits_{j=k}^{\ell_i-1} y_{\tau_j}  & \text{if $i$ is odd}, \\
\sum\limits_{k=\ell_{i-1}}^{\ell_i-1} \prod\limits_{j=\ell_{i-1}+1}^{k} y_{\tau_j}  & \text{if $i$ is even}. \\
\end{cases} 
\end{align*} 
\end{theorem}

\begin{proof}
We proceed by an induction on $n$.  Consider the case that $n=1$, using Lemma \ref{Rab3.3} and the fact that the number of tiles in $\mathcal{G}[a_1]$ is $a_1-1$, we have 
\[
F(\mathcal{G}[a_1])=\sum_{k=1}^{a_1} \prod_{j=k}^{a_1-1} y_{\tau_j}= \varphi_1 = C^{-1}_1  C_1 \varphi_1 = C^{-1}_1 \mathcal{L}_1.
\]
For $n=2$, using (\ref{dual formula}) and the fact that $\ell_1=a_1$ and $C_1=1$, we have 
\begin{align*}
F(\mathcal{G}[a_1,a_2]) & =  y_{34} F(\mathcal{G}[a_1]) F(\mathcal{H}_2) + y_{56} = y_{\tau_{a_1}} (C^{-1} \mathcal{L}_1) \varphi_2 + 1 \\
& = \mathcal{L}_1\mathcal{L}_2+1 =\mathcal{N}[\mathcal{L}_1,\mathcal{L}_2].
\end{align*}

Assume that $n$ is odd and our formula holds for $m<n$. The minimal matching of $\mathcal{H}_n$ is the matching which inherits from $P_{-}$, shown in Figure \ref{Hn minimal perfect matching}, so $F(\mathcal{H}_n)=\varphi_n$. Then by (\ref{dual formula})

\begin{align*}
F(\mathcal{G}[a_1,a_2,\ldots,a_n]) & =F(\mathcal{G}[a_1,a_2,\ldots,a_{n-1}]) F(\mathcal{H}_n) + \prod\nolimits_{j=\ell_{n-2}}^{\ell_n-1} y_{\tau_j} F(\mathcal{G}[a_1,a_2,\ldots,a_{n-2}]) \\
& =  \mathcal{N}[\mathcal{L}_1,\mathcal{L}_2,\ldots,\mathcal{L}_{n-1}] \varphi_n +  \prod\nolimits_{j=\ell_{n-2}}^{\ell_n-1} y_{\tau_j} C^{-1}_{n-2}\mathcal{N}[\mathcal{L}_1,\mathcal{L}_2,\ldots,\mathcal{L}_{n-2}]  \\
& = \mathcal{N}[\mathcal{L}_1,\mathcal{L}_2,\ldots,\mathcal{L}_{n-1}] \varphi_n +  C^{-1}_n\mathcal{N}[\mathcal{L}_1,\mathcal{L}_2,\ldots,\mathcal{L}_{n-2}]  \\
& =  C^{-1}_n (\mathcal{L}_n\mathcal{N}[\mathcal{L}_1,\mathcal{L}_2,\ldots,\mathcal{L}_{n-1}]  + \mathcal{N}[\mathcal{L}_1,\mathcal{L}_2,\ldots,\mathcal{L}_{n-2}] ) \\
& = C^{-1}_n \mathcal{N}[\mathcal{L}_1,\mathcal{L}_2,\ldots,\mathcal{L}_n].
\end{align*}

\begin{figure}[H]
\resizebox{.5\width}{.5\height}{
\begin{minipage}{0.5\linewidth}
\begin{tikzpicture}
\draw[dashed] (1,0) rectangle (2,1);
\draw (2,0) rectangle(3,1);
\draw (2,1) rectangle(3,2);
\draw (3,1) rectangle(4,2);
\draw (3,2) rectangle(4,3);
\draw (5,3) rectangle(6,4);
\node at (4.5,3) {$\pmb{\iddots}$};
\draw[red,ultra thick] (2,0) -- (3,0);
\draw[red,ultra thick] (2,1) -- (2,2);
\draw[red,ultra thick] (3,1) -- (4,1);
\draw[red,ultra thick] (3,2) -- (3,3);
\draw[red,ultra thick] (5,3) -- (6,3);
\draw[red,ultra thick] (5,4) -- (6,4);
\end{tikzpicture}
\end{minipage}
\begin{minipage}{0.5\linewidth}
\begin{tikzpicture}
\draw[dashed] (1,0) rectangle (2,1);
\draw (2,0) rectangle(3,1);
\draw (2,1) rectangle(3,2);
\draw (3,1) rectangle(4,2);
\draw (3,2) rectangle(4,3);
\draw (5,3) rectangle(6,4);
\node at (4.5,3) {$\pmb{\iddots}$};
\draw[red,ultra thick] (2,0) -- (3,0);
\draw[red,ultra thick] (2,1) -- (2,2);
\draw[red,ultra thick] (3,1) -- (4,1);
\draw[red,ultra thick] (3,2) -- (3,3);
\draw[red,ultra thick] (5,3) -- (5,4);
\draw[red,ultra thick] (6,3) -- (6,4);
\end{tikzpicture}
\end{minipage}}
\caption{The minimal perfect matching of $\mathcal{H}_n$. The dashed tile is a sign-changed tile. }\label{Hn minimal perfect matching}
\end{figure}

Assume that $n$ is even and our formula holds for $m<n$. The minimal matching of $\mathcal{H}_n$ is a union of the matching which inherits from $P_{-}$ and $e_{\ell_{n-1}}$, shown in Figure \ref{Hn minimal perfect matching2}, so $F(\mathcal{H}_n)=\varphi_n$. Then by (\ref{dual formula})
\begin{align*}
F(\mathcal{G}[a_1,a_2,\ldots,a_n]) & =y_{\tau_{\ell_{n-1}}} F(\mathcal{G}[a_1,a_2,\ldots,a_{n-1}]) F(\mathcal{H}_n) + F(\mathcal{G}[a_1,a_2,\ldots,a_{n-2}]) \\
& =  y_{\tau_{\ell_{n-1}}} C^{-1}_{n-1}\mathcal{N}[\mathcal{L}_1,\mathcal{L}_2,\ldots,\mathcal{L}_{n-1}] \varphi_n + \mathcal{N}[\mathcal{L}_1,\mathcal{L}_2,\ldots,\mathcal{L}_{n-2}]  \\
& = \mathcal{N}[\mathcal{L}_1,\mathcal{L}_2,\ldots,\mathcal{L}_{n-1}] C_n \varphi_n + \mathcal{N}[\mathcal{L}_1,\mathcal{L}_2,\ldots,\mathcal{L}_{n-2}]  \\
& =  \mathcal{L}_n\mathcal{N}[\mathcal{L}_1,\mathcal{L}_2,\ldots,\mathcal{L}_{n-1}]  + \mathcal{N}[\mathcal{L}_1,\mathcal{L}_2,\ldots,\mathcal{L}_{n-2}] \\
& =\mathcal{N}[\mathcal{L}_1,\mathcal{L}_2,\ldots,\mathcal{L}_n].
\end{align*}
\begin{figure}[H]
\resizebox{.5\width}{.5\height}{
\begin{minipage}{0.5\linewidth}
\begin{tikzpicture}
\draw[dashed] (1,0) rectangle (2,1);
\draw (2,0) rectangle(3,1);
\draw (2,1) rectangle(3,2);
\draw (3,1) rectangle(4,2);
\draw (3,2) rectangle(4,3);
\draw (5,3) rectangle(6,4);
\node at (4.5,3) {$\pmb{\iddots}$};
\draw[red,ultra thick] (2,0) -- (2,1);
\draw[red,ultra thick] (3,0) -- (3,1);
\draw[red,ultra thick] (2,2) -- (3,2);
\draw[red,ultra thick] (4,1) -- (4,2);
\draw[red,ultra thick] (3,3) -- (4,3);
\draw[red,ultra thick] (5,4) -- (6,4);
\end{tikzpicture}
\end{minipage}
\begin{minipage}{0.5\linewidth}
\begin{tikzpicture}
\draw[dashed] (1,0) rectangle (2,1);
\draw (2,0) rectangle(3,1);
\draw (2,1) rectangle(3,2);
\draw (3,1) rectangle(4,2);
\draw (3,2) rectangle(4,3);
\draw (5,3) rectangle(6,4);
\node at (4.5,3) {$\pmb{\iddots}$};
\draw[red,ultra thick] (2,0) -- (2,1);
\draw[red,ultra thick] (3,0) -- (3,1);
\draw[red,ultra thick] (2,2) -- (3,2);
\draw[red,ultra thick] (4,1) -- (4,2);
\draw[red,ultra thick] (3,3) -- (4,3);
\draw[red,ultra thick] (6,3) -- (6,4);
\end{tikzpicture}
\end{minipage}}
\caption{The minimal perfect matching of $\mathcal{H}_n$. The dashed tile is a sign-changed tile. }\label{Hn minimal perfect matching2}
\end{figure}
\end{proof}

In practice, we always assume that $a_1>1$ for any continued fraction $[a_1,a_2,\ldots,a_n]$. Otherwise, we take the rotation of a snake graph by 180 degrees or the flips at the lines $y=-x$,  it follows from \cite[Proposition 3.1 (b) and (c)]{CS20} that $\mathcal{G}[a_1,a_2,\ldots,a_n]\cong \mathcal{G}[a_n,\ldots,a_2,a_1]$. In addition, $[a_1,a_2,\ldots,a_{n-1},1]=[a_1,a_2,\ldots,a_{n-1}+1]$.

\begin{lemma}\label{Lemma36}
Let $\mathcal{G}$ be the snake graph associated to the HL-module  parameterized by $\alpha_{i,j}$ and $F=\sum_{P\in {\rm Match}(\mathcal{G})} y(P)$ be its F-polynomial. Then
\[
F|_{\mathbb{P}} = y(P_{+})|_{\mathbb{P}}= (y_iy_{i+1} \cdots y_j)|_{\mathbb{P}}.
\]
\end{lemma}

\begin{proof}
Following \cite{MSW11}, all perfect matchings in $\text{Match}(\mathcal{G})$ form a poset with the minimal perfect matching $P_{-}$ and the maximal perfect matching $P_{+}$. 
Recall that in \cite{Rab18} a tile $G$ can be \textit{turned} if two of its edges are in a perfect matching $P$. Let $P'$ be the perfect matching obtained by replacing the two edges of $G$ in $P$ with the other two edges of $G$. We say that $P'$ is obtained from $P$ by turning the tile $G$. Every perfect matching is obtained from $P_{-}$ or $P_{+}$ by turning a sequence of tiles. 

The height monomial $y(P')$ of $P'$ is defined recursively by $y(P_{-}) = 1$ and if $P'$ is above $P$ and obtained by turning a tile $G_\ell$ then $y(P')=y_\ell y(P)$. In the following, we will prove that for each step of turning a tile, $y(P)|_{\mathbb{P}}\leq y(P')|_{\mathbb{P}}$, meaning that $y(P')|_{\mathbb{P}}$ contains all possible factors with multiplicities appearing in $y(P)|_{\mathbb{P}}$. As a conclusion, $F|_\mathbb{P} = y(P_{+})|_\mathbb{P}$.

In our setting, if $P$ is above $P_{-}$ and obtained by turning a tile $G_\ell$, then $y(P)=y_{\ell}$ and either $\ell\not\in \{i,j\}$ is a source or $\ell\in \{i,j\}$ such that $\ell \to l$ for a certain $l\in I$. In general, if $P'$ is above $P$ and obtained by turning a tile $G_\ell$ then $y(P')=y_{\ell}y(P)$, and at least one of the following two cases occur:
\begin{itemize}
\item[(i)] either $\ell\not\in \{i,j\}$ is a source or $\ell\in \{i,j\}$ such that $\ell \to l$ for a certain $l\in I$, where $y_l$ is not a factor in $y(P)$,
\item[(ii)] there exists a certain $l\in I$ (not unique) such that $l\to \ell$ in $Q_\xi$, where $y_l$ is a factor in $y(P)$. 
\end{itemize}

For the case (i), $y_\ell= x'_\ell$ or all possible factors in the denominator of $y_\ell$ cannot be cancelled by $y(P)$, and hence  $y(P)|_{\mathbb{P}}\leq y(P')|_{\mathbb{P}}$. For the case (ii), all possible arrows near the vertex $\ell$ are shown as follows.
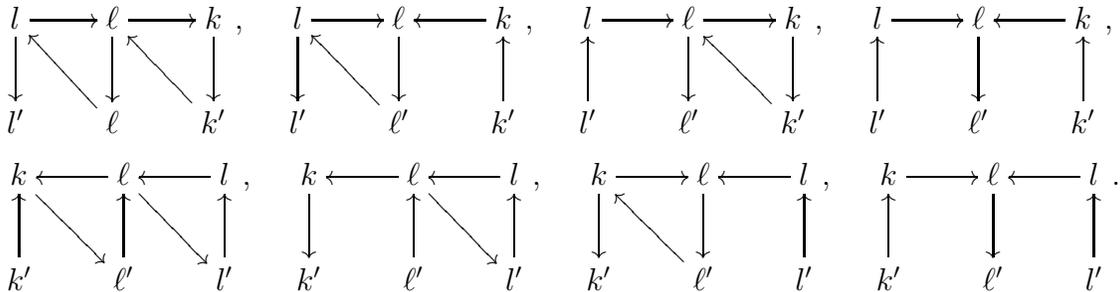
\begin{align*}
& \xymatrix{
l \ar[r] \ar[d]  & \ell \ar[r] \ar[d] & k \ar[d] \\
l' & \ell  \ar[ul] & k' \ar[ul] }, 
\quad
\xymatrix{
l \ar[r] \ar[d]  & \ell \ar[d] & k \ar[l]  \\
l' & \ell'  \ar[ul] & k' \ar[u] },
 \quad
\xymatrix{
l \ar[r] & \ell \ar[r] \ar[d] & k \ar[d] \\
l' \ar[u]  & \ell' &  k' \ar[ul] },  
\quad 
\xymatrix{
l \ar[r] & \ell \ar[d] & k \ar[l] \\
l' \ar[u]  & \ell' &  k' \ar[u] }, 
\\
& \xymatrix{
k \ar[dr] & \ell \ar[l] \ar[dr] & l \ar[l] \\
k'\ar[u] & \ell' \ar[u] & l' \ar[u] },  
\quad
\xymatrix{
k \ar[d] & \ell \ar[l] \ar[dr] & l \ar[l] \\
k' & \ell' \ar[u] & l' \ar[u] },  
\quad 
\xymatrix{
k \ar[d] \ar[r] & \ell  \ar[d] & l \ar[l] \\
k' & \ell' \ar[ul] & l' \ar[u] },  
\quad
\xymatrix{
k  \ar[r] & \ell  \ar[d] & l \ar[l] \\
k' \ar[u] & \ell' & l' \ar[u] }.
\end{align*}
So checking case by case from up to down and from left to right
\begin{align*}
& y(P')|_{\mathbb{P}}=(\frac{x'_k}{x'_\ell}y(P))|_{\mathbb{P}}= y(P)|_{\mathbb{P}}, \text{ because $x'_\ell$ is cancelled by the numerator in $y_l$}, \\
& y(P')|_{\mathbb{P}}=(\frac{1}{x'_\ell }y(P))|_{\mathbb{P}}= y(P)|_{\mathbb{P}}, \text{ because $x'_\ell$ is cancelled by the numerator in $y_l$}, \\
& y(P')|_{\mathbb{P}}=(\frac{x'_k}{x'_\ell}y(P))|_{\mathbb{P}}= x'^{-1}_\ell y(P)|_{\mathbb{P}}, \\
& y(P')|_{\mathbb{P}}=(\frac{1}{x'_\ell}y(P))|_{\mathbb{P}}= x'^{-1}_\ell y(P)|_{\mathbb{P}}, \\
& y(P')|_{\mathbb{P}}=(\frac{x'_\ell}{x'_l}y(P))|_{\mathbb{P}}= y(P)|_{\mathbb{P}}, \text{ because $x'_l$ is cancelled by the numerator in $y_l$}, \\
& y(P')|_{\mathbb{P}}=(\frac{x'_\ell}{x'_l}y(P))|_{\mathbb{P}}= y(P)|_{\mathbb{P}}, \text{ because $x'_l$ is cancelled by the numerator in $y_l$}, \\
& y(P')|_{\mathbb{P}}=(\frac{1}{x'_\ell}y(P))|_{\mathbb{P}}= y(P)|_{\mathbb{P}}, \text{ because $x'_\ell$ is cancelled by the numerator in $y_k$}, \\
& y(P')|_{\mathbb{P}}=(\frac{1}{x'_\ell}y(P))|_{\mathbb{P}}= x'^{-1}_\ell y(P)|_{\mathbb{P}},
\end{align*}
the seventh equation is because both $G_l$ and $G_k$ have been turned from $P_{-}$ before turning the tile $G_\ell$. 
\end{proof}

\begin{theorem}\label{snake graph formula}
Let $\mathscr{C}_\xi$ be the full subcategory introduced by Brito and Chari \cite{BC19} and $\mathcal{G}$ be the labeled snake graph associated to $x[\alpha_{i,j}]$. Then 
\begin{align*}
x[\alpha_{i,j}]  = \frac{1}{ \prod_{\ell=i}^j x_\ell}  \left(\frac{\sum_{P\in {\rm Match}(\mathcal{G})}  x(P) y(P) }{\bigoplus_{P\in {\rm Match}(\mathcal{G})} y(P)}\right),
\end{align*}
where the sign $\oplus$ appearing in the denominator refers to the (auxiliary) addition in a tropical semifield $\mathbb{P}={\rm Trop}(u_1,u_2,\ldots,u_r)$ $(r\in \mathbb{Z}_{\geq 1})$ and $y_\ell = \prod_{j \to \ell} u_j \prod_{\ell \to j} u^{-1}_j$ is a Laurent monomial in $r$ variables $u_1,u_2,\ldots,u_r$ for $i\leq \ell \leq j$. In particular, 
\begin{align*}
\chi_q(\iota(x[\alpha_{i,j}])) = \frac{1}{ \prod_{\ell=i}^j \chi_q(\iota(x_\ell))} \chi_q\left(\iota(\frac{\sum_{P\in {\rm Match}(\mathcal{G})}  x(P) y(P) }{\bigoplus_{P\in {\rm Match}(\mathcal{G})} y(P)})\right).
\end{align*}
\end{theorem}
\begin{proof}
The proof of the first part follows \cite[Theorem 4.10]{MSW11} and \cite[Theorem 3.7]{FZ07} (see Theorem \ref{FZ07thm3.7}). The second one follows from \cite[Theorem 1, Corollary in Section 1.3]{BC19} (see Theorem \ref{BCTheorem1}).
\end{proof}

We explain Theorem \ref{snake graph formula} by the following example.
\begin{example}\label{exmaple39}
Continue our previous Examples \ref{example34}. Since $x_8,x_9$ are not mutated, our coefficients are in $\text{Trop}(x'_1,x'_2,x'_3,x'_4,x'_5,x'_6,x'_7,x'_8,x'_9,x_8,x_9)$. By Theorem \ref{snake graph formula}, we substitute variables by 
\begin{align*}
& y_1 = \frac{x'_1}{x'_2}, \quad y_2 = x'_2, \quad y_3 = \frac{x'_4}{x'_3}, \quad y_4 = \frac{x'_5}{x'_4}, \quad y_5 = \frac{1}{x'_5}, \quad  y_6 = \frac{x'_6}{x'_7}, \quad  y_7 = \frac{x'_7 x_8}{x'_8}.  
\end{align*}

All possible perfect matchings of the labeled snake graph associated to the HL-module $L(Y_{1,-3}Y_{3,-7}Y_{6,-2}Y_{8,-6})$ are shown in Table \ref{perfect matchings x17}. In Table \ref{perfect matchings x17} we list the monomial $x(P)y(P)$ in the below of $P\in\text{Match}(\mathcal{G})$.

\begin{table}
\center
\begin{tabular}{|c|c|c|c|}
\hline
\resizebox{.5\width}{.5\height}{
\begin{minipage}{0.4\linewidth}
\begin{tikzpicture}
\draw (0,0) rectangle (1,1);
\draw (1,0) rectangle (2,1);
\draw (2,0) rectangle(3,1);
\draw (2,1) rectangle(3,2);
\draw (3,1) rectangle(4,2);
\draw (4,1) rectangle(5,2);
\draw (4,2) rectangle(5,3);
\node at (0.5,0.5) {1};
\node at (1.5,0.5) {2};
\node at (2.5,0.5) {3};
\node at (2.5,1.5) {4};
\node at (3.5,1.5) {5};
\node at (4.5,1.5) {6};
\node at (4.5,2.5) {7};
\node[above] at (0.5,0.9) {2};
\node[above] at (1.5,0.9) {3};
\node[below] at (1.5,0.1) {1};
\node[below] at (2.5,0.1) {2};
\node[left] at (2.1,1.5) {3};
\node[right] at (2.9,0.5) {4};
\node[below] at (3.5,1.1) {4};
\node[above] at (2.5,1.9) {5};
\node[below] at (4.5,1.1) {5};
\node[above] at (3.5,1.9) {6};
\node[left] at (4.1,2.5) {6};
\node[right] at (4.9,1.5) {7};
\draw[red,ultra thick] (0,0) -- (0,1);
\draw[red,ultra thick] (1,0) -- (2,0);
\draw[red,ultra thick] (1,1) -- (2,1);
\draw[red,ultra thick] (3,0) -- (3,1);
\draw[red,ultra thick] (2,2) -- (3,2);
\draw[red,ultra thick] (4,1) -- (5,1);
\draw[red,ultra thick] (4,2) -- (4,3);
\draw[red,ultra thick] (5,2) -- (5,3);
\end{tikzpicture}
\caption*{$x_1x_3x_4 x^2_5 x_6$}
\end{minipage}} & 
\resizebox{.5\width}{.5\height}{
\begin{minipage}{0.4\linewidth}
\begin{tikzpicture}
\draw (0,0) rectangle (1,1);
\draw (1,0) rectangle (2,1);
\draw (2,0) rectangle(3,1);
\draw (2,1) rectangle(3,2);
\draw (3,1) rectangle(4,2);
\draw (4,1) rectangle(5,2);
\draw (4,2) rectangle(5,3);
\node at (0.5,0.5) {1};
\node at (1.5,0.5) {2};
\node at (2.5,0.5) {3};
\node at (2.5,1.5) {4};
\node at (3.5,1.5) {5};
\node at (4.5,1.5) {6};
\node at (4.5,2.5) {7};
\node[above] at (0.5,0.9) {2};
\node[above] at (1.5,0.9) {3};
\node[below] at (1.5,0.1) {1};
\node[below] at (2.5,0.1) {2};
\node[left] at (2.1,1.5) {3};
\node[right] at (2.9,0.5) {4};
\node[below] at (3.5,1.1) {4};
\node[above] at (2.5,1.9) {5};
\node[below] at (4.5,1.1) {5};
\node[above] at (3.5,1.9) {6};
\node[left] at (4.1,2.5) {6};
\node[right] at (4.9,1.5) {7};
\draw[red,ultra thick] (0,0) -- (0,1);
\draw[red,ultra thick] (1,0) -- (2,0);
\draw[red,ultra thick] (1,1) -- (2,1);
\draw[red,ultra thick] (3,0) -- (3,1);
\draw[red,ultra thick] (2,2) -- (3,2);
\draw[red,ultra thick] (4,1) -- (5,1);
\draw[red,ultra thick] (4,2) -- (5,2);
\draw[red,ultra thick] (4,3) -- (5,3);
\end{tikzpicture}
\caption*{$x_1x_3x_4 x^2_5 y_7$}
\end{minipage}} & 
\resizebox{.5\width}{.5\height}{
\begin{minipage}{0.4\linewidth}
\begin{tikzpicture}
\draw (0,0) rectangle (1,1);
\draw (1,0) rectangle (2,1);
\draw (2,0) rectangle(3,1);
\draw (2,1) rectangle(3,2);
\draw (3,1) rectangle(4,2);
\draw (4,1) rectangle(5,2);
\draw (4,2) rectangle(5,3);
\node at (0.5,0.5) {1};
\node at (1.5,0.5) {2};
\node at (2.5,0.5) {3};
\node at (2.5,1.5) {4};
\node at (3.5,1.5) {5};
\node at (4.5,1.5) {6};
\node at (4.5,2.5) {7};
\node[above] at (0.5,0.9) {2};
\node[above] at (1.5,0.9) {3};
\node[below] at (1.5,0.1) {1};
\node[below] at (2.5,0.1) {2};
\node[left] at (2.1,1.5) {3};
\node[right] at (2.9,0.5) {4};
\node[below] at (3.5,1.1) {4};
\node[above] at (2.5,1.9) {5};
\node[below] at (4.5,1.1) {5};
\node[above] at (3.5,1.9) {6};
\node[left] at (4.1,2.5) {6};
\node[right] at (4.9,1.5) {7};
\draw[red,ultra thick] (0,0) -- (0,1);
\draw[red,ultra thick] (1,0) -- (2,0);
\draw[red,ultra thick] (1,1) -- (2,1);
\draw[red,ultra thick] (3,0) -- (3,1);
\draw[red,ultra thick] (2,2) -- (3,2);
\draw[red,ultra thick] (4,1) -- (4,2);
\draw[red,ultra thick] (5,1) -- (5,2);
\draw[red,ultra thick] (4,3) -- (5,3);
\end{tikzpicture}
\caption*{$x_1x_3x_4 x_5 x_7 y_6y_7$}
\end{minipage}}  &
\resizebox{.5\width}{.5\height}{
\begin{minipage}{0.4\linewidth}
\begin{tikzpicture}
\draw (0,0) rectangle (1,1);
\draw (1,0) rectangle (2,1);
\draw (2,0) rectangle(3,1);
\draw (2,1) rectangle(3,2);
\draw (3,1) rectangle(4,2);
\draw (4,1) rectangle(5,2);
\draw (4,2) rectangle(5,3);
\node at (0.5,0.5) {1};
\node at (1.5,0.5) {2};
\node at (2.5,0.5) {3};
\node at (2.5,1.5) {4};
\node at (3.5,1.5) {5};
\node at (4.5,1.5) {6};
\node at (4.5,2.5) {7};
\node[above] at (0.5,0.9) {2};
\node[above] at (1.5,0.9) {3};
\node[below] at (1.5,0.1) {1};
\node[below] at (2.5,0.1) {2};
\node[left] at (2.1,1.5) {3};
\node[right] at (2.9,0.5) {4};
\node[below] at (3.5,1.1) {4};
\node[above] at (2.5,1.9) {5};
\node[below] at (4.5,1.1) {5};
\node[above] at (3.5,1.9) {6};
\node[left] at (4.1,2.5) {6};
\node[right] at (4.9,1.5) {7};
\draw[red,ultra thick] (0,0) -- (0,1);
\draw[red,ultra thick] (1,0) -- (1,1);
\draw[red,ultra thick] (2,0) -- (2,1);
\draw[red,ultra thick] (3,0) -- (3,1);
\draw[red,ultra thick] (2,2) -- (3,2);
\draw[red,ultra thick] (4,1) -- (5,1);
\draw[red,ultra thick] (4,2) -- (4,3);
\draw[red,ultra thick] (5,2) -- (5,3);
\end{tikzpicture}
\caption*{$x_4 x^2_5 x_6 y_2$}
\end{minipage}}    \\
\hline
\resizebox{.5\width}{.5\height}{
\begin{minipage}{0.4\linewidth}
\begin{tikzpicture}
\draw (0,0) rectangle (1,1);
\draw (1,0) rectangle (2,1);
\draw (2,0) rectangle(3,1);
\draw (2,1) rectangle(3,2);
\draw (3,1) rectangle(4,2);
\draw (4,1) rectangle(5,2);
\draw (4,2) rectangle(5,3);
\node at (0.5,0.5) {1};
\node at (1.5,0.5) {2};
\node at (2.5,0.5) {3};
\node at (2.5,1.5) {4};
\node at (3.5,1.5) {5};
\node at (4.5,1.5) {6};
\node at (4.5,2.5) {7};
\node[above] at (0.5,0.9) {2};
\node[above] at (1.5,0.9) {3};
\node[below] at (1.5,0.1) {1};
\node[below] at (2.5,0.1) {2};
\node[left] at (2.1,1.5) {3};
\node[right] at (2.9,0.5) {4};
\node[below] at (3.5,1.1) {4};
\node[above] at (2.5,1.9) {5};
\node[below] at (4.5,1.1) {5};
\node[above] at (3.5,1.9) {6};
\node[left] at (4.1,2.5) {6};
\node[right] at (4.9,1.5) {7};
\draw[red,ultra thick] (0,0) -- (0,1);
\draw[red,ultra thick] (1,0) -- (1,1);
\draw[red,ultra thick] (2,0) -- (2,1);
\draw[red,ultra thick] (3,0) -- (3,1);
\draw[red,ultra thick] (2,2) -- (3,2);
\draw[red,ultra thick] (4,1) -- (5,1);
\draw[red,ultra thick] (4,2) -- (5,2);
\draw[red,ultra thick] (4,3) -- (5,3);
\end{tikzpicture}
\caption*{$x_4 x^2_5 y_2y_7$} 
\end{minipage}} & 
\resizebox{.5\width}{.5\height}{
\begin{minipage}{0.4\linewidth}
\begin{tikzpicture}
\draw (0,0) rectangle (1,1);
\draw (1,0) rectangle (2,1);
\draw (2,0) rectangle(3,1);
\draw (2,1) rectangle(3,2);
\draw (3,1) rectangle(4,2);
\draw (4,1) rectangle(5,2);
\draw (4,2) rectangle(5,3);
\node at (0.5,0.5) {1};
\node at (1.5,0.5) {2};
\node at (2.5,0.5) {3};
\node at (2.5,1.5) {4};
\node at (3.5,1.5) {5};
\node at (4.5,1.5) {6};
\node at (4.5,2.5) {7};
\node[above] at (0.5,0.9) {2};
\node[above] at (1.5,0.9) {3};
\node[below] at (1.5,0.1) {1};
\node[below] at (2.5,0.1) {2};
\node[left] at (2.1,1.5) {3};
\node[right] at (2.9,0.5) {4};
\node[below] at (3.5,1.1) {4};
\node[above] at (2.5,1.9) {5};
\node[below] at (4.5,1.1) {5};
\node[above] at (3.5,1.9) {6};
\node[left] at (4.1,2.5) {6};
\node[right] at (4.9,1.5) {7};
\draw[red,ultra thick] (0,0) -- (0,1);
\draw[red,ultra thick] (1,0) -- (1,1);
\draw[red,ultra thick] (2,0) -- (2,1);
\draw[red,ultra thick] (3,0) -- (3,1);
\draw[red,ultra thick] (2,2) -- (3,2);
\draw[red,ultra thick] (4,1) -- (4,2);
\draw[red,ultra thick] (5,1) -- (5,2);
\draw[red,ultra thick] (4,3) -- (5,3);
\end{tikzpicture}
\caption*{$x_4 x_5 x_7 y_2y_6y_7$} 
\end{minipage}} &
\resizebox{.5\width}{.5\height}{ 
\begin{minipage}{0.4\linewidth}
\begin{tikzpicture}
\draw (0,0) rectangle (1,1);
\draw (1,0) rectangle (2,1);
\draw (2,0) rectangle(3,1);
\draw (2,1) rectangle(3,2);
\draw (3,1) rectangle(4,2);
\draw (4,1) rectangle(5,2);
\draw (4,2) rectangle(5,3);
\node at (0.5,0.5) {1};
\node at (1.5,0.5) {2};
\node at (2.5,0.5) {3};
\node at (2.5,1.5) {4};
\node at (3.5,1.5) {5};
\node at (4.5,1.5) {6};
\node at (4.5,2.5) {7};
\node[above] at (0.5,0.9) {2};
\node[above] at (1.5,0.9) {3};
\node[below] at (1.5,0.1) {1};
\node[below] at (2.5,0.1) {2};
\node[left] at (2.1,1.5) {3};
\node[right] at (2.9,0.5) {4};
\node[below] at (3.5,1.1) {4};
\node[above] at (2.5,1.9) {5};
\node[below] at (4.5,1.1) {5};
\node[above] at (3.5,1.9) {6};
\node[left] at (4.1,2.5) {6};
\node[right] at (4.9,1.5) {7};
\draw[red,ultra thick] (0,0) -- (1,0);
\draw[red,ultra thick] (0,1) -- (1,1);
\draw[red,ultra thick] (2,0) -- (2,1);
\draw[red,ultra thick] (3,0) -- (3,1);
\draw[red,ultra thick] (2,2) -- (3,2);
\draw[red,ultra thick] (4,1) -- (5,1);
\draw[red,ultra thick] (4,2) -- (4,3);
\draw[red,ultra thick] (5,2) -- (5,3);
\end{tikzpicture}
\caption*{$x_2 x_4 x^2_5 x_6 y_1y_2$} 
\end{minipage}} & 
\resizebox{.5\width}{.5\height}{
\begin{minipage}{0.4\linewidth}
\begin{tikzpicture}
\draw (0,0) rectangle (1,1);
\draw (1,0) rectangle (2,1);
\draw (2,0) rectangle(3,1);
\draw (2,1) rectangle(3,2);
\draw (3,1) rectangle(4,2);
\draw (4,1) rectangle(5,2);
\draw (4,2) rectangle(5,3);
\node at (0.5,0.5) {1};
\node at (1.5,0.5) {2};
\node at (2.5,0.5) {3};
\node at (2.5,1.5) {4};
\node at (3.5,1.5) {5};
\node at (4.5,1.5) {6};
\node at (4.5,2.5) {7};
\node[above] at (0.5,0.9) {2};
\node[above] at (1.5,0.9) {3};
\node[below] at (1.5,0.1) {1};
\node[below] at (2.5,0.1) {2};
\node[left] at (2.1,1.5) {3};
\node[right] at (2.9,0.5) {4};
\node[below] at (3.5,1.1) {4};
\node[above] at (2.5,1.9) {5};
\node[below] at (4.5,1.1) {5};
\node[above] at (3.5,1.9) {6};
\node[left] at (4.1,2.5) {6};
\node[right] at (4.9,1.5) {7};
\draw[red,ultra thick] (0,0) -- (0,1);
\draw[red,ultra thick] (1,0) -- (1,1);
\draw[red,ultra thick] (2,0) -- (3,0);
\draw[red,ultra thick] (2,1) -- (3,1);
\draw[red,ultra thick] (2,2) -- (3,2);
\draw[red,ultra thick] (4,1) -- (5,1);
\draw[red,ultra thick] (4,2) -- (4,3);
\draw[red,ultra thick] (5,2) -- (5,3);
\end{tikzpicture}
\caption*{$x_2 x^2_5 x_6 y_2y_3$}
\end{minipage}}\\
\hline
\resizebox{.5\width}{.5\height}{
\begin{minipage}{0.4\linewidth}
\begin{tikzpicture}
\draw (0,0) rectangle (1,1);
\draw (1,0) rectangle (2,1);
\draw (2,0) rectangle(3,1);
\draw (2,1) rectangle(3,2);
\draw (3,1) rectangle(4,2);
\draw (4,1) rectangle(5,2);
\draw (4,2) rectangle(5,3);
\node at (0.5,0.5) {1};
\node at (1.5,0.5) {2};
\node at (2.5,0.5) {3};
\node at (2.5,1.5) {4};
\node at (3.5,1.5) {5};
\node at (4.5,1.5) {6};
\node at (4.5,2.5) {7};
\node[above] at (0.5,0.9) {2};
\node[above] at (1.5,0.9) {3};
\node[below] at (1.5,0.1) {1};
\node[below] at (2.5,0.1) {2};
\node[left] at (2.1,1.5) {3};
\node[right] at (2.9,0.5) {4};
\node[below] at (3.5,1.1) {4};
\node[above] at (2.5,1.9) {5};
\node[below] at (4.5,1.1) {5};
\node[above] at (3.5,1.9) {6};
\node[left] at (4.1,2.5) {6};
\node[right] at (4.9,1.5) {7};
\draw[red,ultra thick] (0,0) -- (1,0);
\draw[red,ultra thick] (0,1) -- (1,1);
\draw[red,ultra thick] (2,0) -- (3,0);
\draw[red,ultra thick] (2,1) -- (3,1);
\draw[red,ultra thick] (2,2) -- (3,2);
\draw[red,ultra thick] (4,1) -- (5,1);
\draw[red,ultra thick] (4,2) -- (4,3);
\draw[red,ultra thick] (5,2) -- (5,3);
\end{tikzpicture}
\caption*{$x^2_2 x^2_5 x_6 y_1y_2y_3$}
\end{minipage}} & 
\resizebox{.5\width}{.5\height}{
\begin{minipage}{0.4\linewidth}
\begin{tikzpicture}
\draw (0,0) rectangle (1,1);
\draw (1,0) rectangle (2,1);
\draw (2,0) rectangle(3,1);
\draw (2,1) rectangle(3,2);
\draw (3,1) rectangle(4,2);
\draw (4,1) rectangle(5,2);
\draw (4,2) rectangle(5,3);
\node at (0.5,0.5) {1};
\node at (1.5,0.5) {2};
\node at (2.5,0.5) {3};
\node at (2.5,1.5) {4};
\node at (3.5,1.5) {5};
\node at (4.5,1.5) {6};
\node at (4.5,2.5) {7};
\node[above] at (0.5,0.9) {2};
\node[above] at (1.5,0.9) {3};
\node[below] at (1.5,0.1) {1};
\node[below] at (2.5,0.1) {2};
\node[left] at (2.1,1.5) {3};
\node[right] at (2.9,0.5) {4};
\node[below] at (3.5,1.1) {4};
\node[above] at (2.5,1.9) {5};
\node[below] at (4.5,1.1) {5};
\node[above] at (3.5,1.9) {6};
\node[left] at (4.1,2.5) {6};
\node[right] at (4.9,1.5) {7};
\draw[red,ultra thick] (0,0) -- (1,0);
\draw[red,ultra thick] (0,1) -- (1,1);
\draw[red,ultra thick] (2,0) -- (2,1);
\draw[red,ultra thick] (3,0) -- (3,1);
\draw[red,ultra thick] (2,2) -- (3,2);
\draw[red,ultra thick] (4,1) -- (5,1);
\draw[red,ultra thick] (4,2) -- (5,2);
\draw[red,ultra thick] (4,3) -- (5,3);
\end{tikzpicture}
\caption*{$x_2 x_4 x^2_5 y_1y_2y_7$} 
\end{minipage}} & 
\resizebox{.5\width}{.5\height}{
\begin{minipage}{0.4\linewidth}
\begin{tikzpicture}
\draw (0,0) rectangle (1,1);
\draw (1,0) rectangle (2,1);
\draw (2,0) rectangle(3,1);
\draw (2,1) rectangle(3,2);
\draw (3,1) rectangle(4,2);
\draw (4,1) rectangle(5,2);
\draw (4,2) rectangle(5,3);
\node at (0.5,0.5) {1};
\node at (1.5,0.5) {2};
\node at (2.5,0.5) {3};
\node at (2.5,1.5) {4};
\node at (3.5,1.5) {5};
\node at (4.5,1.5) {6};
\node at (4.5,2.5) {7};
\node[above] at (0.5,0.9) {2};
\node[above] at (1.5,0.9) {3};
\node[below] at (1.5,0.1) {1};
\node[below] at (2.5,0.1) {2};
\node[left] at (2.1,1.5) {3};
\node[right] at (2.9,0.5) {4};
\node[below] at (3.5,1.1) {4};
\node[above] at (2.5,1.9) {5};
\node[below] at (4.5,1.1) {5};
\node[above] at (3.5,1.9) {6};
\node[left] at (4.1,2.5) {6};
\node[right] at (4.9,1.5) {7};
\draw[red,ultra thick] (0,0) -- (0,1);
\draw[red,ultra thick] (1,0) -- (1,1);
\draw[red,ultra thick] (2,0) -- (3,0);
\draw[red,ultra thick] (2,1) -- (3,1);
\draw[red,ultra thick] (2,2) -- (3,2);
\draw[red,ultra thick] (4,1) -- (5,1);
\draw[red,ultra thick] (4,2) -- (5,2);
\draw[red,ultra thick] (4,3) -- (5,3);
\end{tikzpicture}
\caption*{$x_2 x^2_5 y_2y_3y_7$}
\end{minipage}} & 
\resizebox{.5\width}{.5\height}{
\begin{minipage}{0.4\linewidth}
\begin{tikzpicture}
\draw (0,0) rectangle (1,1);
\draw (1,0) rectangle (2,1);
\draw (2,0) rectangle(3,1);
\draw (2,1) rectangle(3,2);
\draw (3,1) rectangle(4,2);
\draw (4,1) rectangle(5,2);
\draw (4,2) rectangle(5,3);
\node at (0.5,0.5) {1};
\node at (1.5,0.5) {2};
\node at (2.5,0.5) {3};
\node at (2.5,1.5) {4};
\node at (3.5,1.5) {5};
\node at (4.5,1.5) {6};
\node at (4.5,2.5) {7};
\node[above] at (0.5,0.9) {2};
\node[above] at (1.5,0.9) {3};
\node[below] at (1.5,0.1) {1};
\node[below] at (2.5,0.1) {2};
\node[left] at (2.1,1.5) {3};
\node[right] at (2.9,0.5) {4};
\node[below] at (3.5,1.1) {4};
\node[above] at (2.5,1.9) {5};
\node[below] at (4.5,1.1) {5};
\node[above] at (3.5,1.9) {6};
\node[left] at (4.1,2.5) {6};
\node[right] at (4.9,1.5) {7};
\draw[red,ultra thick] (0,0) -- (1,0);
\draw[red,ultra thick] (0,1) -- (1,1);
\draw[red,ultra thick] (2,0) -- (3,0);
\draw[red,ultra thick] (2,1) -- (3,1);
\draw[red,ultra thick] (2,2) -- (3,2);
\draw[red,ultra thick] (4,1) -- (5,1);
\draw[red,ultra thick] (4,2) -- (5,2);
\draw[red,ultra thick] (4,3) -- (5,3);
\end{tikzpicture}
\caption*{$x^2_2 x^2_5 y_1y_2y_3y_7$}
\end{minipage}} \\
\hline
\resizebox{.5\width}{.5\height}{
\begin{minipage}{0.4\linewidth}
\begin{tikzpicture}
\draw (0,0) rectangle (1,1);
\draw (1,0) rectangle (2,1);
\draw (2,0) rectangle(3,1);
\draw (2,1) rectangle(3,2);
\draw (3,1) rectangle(4,2);
\draw (4,1) rectangle(5,2);
\draw (4,2) rectangle(5,3);
\node at (0.5,0.5) {1};
\node at (1.5,0.5) {2};
\node at (2.5,0.5) {3};
\node at (2.5,1.5) {4};
\node at (3.5,1.5) {5};
\node at (4.5,1.5) {6};
\node at (4.5,2.5) {7};
\node[above] at (0.5,0.9) {2};
\node[above] at (1.5,0.9) {3};
\node[below] at (1.5,0.1) {1};
\node[below] at (2.5,0.1) {2};
\node[left] at (2.1,1.5) {3};
\node[right] at (2.9,0.5) {4};
\node[below] at (3.5,1.1) {4};
\node[above] at (2.5,1.9) {5};
\node[below] at (4.5,1.1) {5};
\node[above] at (3.5,1.9) {6};
\node[left] at (4.1,2.5) {6};
\node[right] at (4.9,1.5) {7};
\draw[red,ultra thick] (0,0) -- (1,0);
\draw[red,ultra thick] (0,1) -- (1,1);
\draw[red,ultra thick] (2,0) -- (2,1);
\draw[red,ultra thick] (3,0) -- (3,1);
\draw[red,ultra thick] (2,2) -- (3,2);
\draw[red,ultra thick] (4,1) -- (4,2);
\draw[red,ultra thick] (5,1) -- (5,2);
\draw[red,ultra thick] (4,3) -- (5,3);
\end{tikzpicture}
\caption*{$x_2 x_4 x_5 x_7 y_1y_2y_6y_7$} 
\end{minipage}} & 
\resizebox{.5\width}{.5\height}{
\begin{minipage}{0.4\linewidth}
\begin{tikzpicture}
\draw (0,0) rectangle (1,1);
\draw (1,0) rectangle (2,1);
\draw (2,0) rectangle(3,1);
\draw (2,1) rectangle(3,2);
\draw (3,1) rectangle(4,2);
\draw (4,1) rectangle(5,2);
\draw (4,2) rectangle(5,3);
\node at (0.5,0.5) {1};
\node at (1.5,0.5) {2};
\node at (2.5,0.5) {3};
\node at (2.5,1.5) {4};
\node at (3.5,1.5) {5};
\node at (4.5,1.5) {6};
\node at (4.5,2.5) {7};
\node[above] at (0.5,0.9) {2};
\node[above] at (1.5,0.9) {3};
\node[below] at (1.5,0.1) {1};
\node[below] at (2.5,0.1) {2};
\node[left] at (2.1,1.5) {3};
\node[right] at (2.9,0.5) {4};
\node[below] at (3.5,1.1) {4};
\node[above] at (2.5,1.9) {5};
\node[below] at (4.5,1.1) {5};
\node[above] at (3.5,1.9) {6};
\node[left] at (4.1,2.5) {6};
\node[right] at (4.9,1.5) {7};
\draw[red,ultra thick] (0,0) -- (0,1);
\draw[red,ultra thick] (1,0) -- (1,1);
\draw[red,ultra thick] (2,0) -- (3,0);
\draw[red,ultra thick] (2,1) -- (3,1);
\draw[red,ultra thick] (2,2) -- (3,2);
\draw[red,ultra thick] (4,1) -- (4,2);
\draw[red,ultra thick] (5,1) -- (5,2);
\draw[red,ultra thick] (4,3) -- (5,3);
\end{tikzpicture}
\caption*{$x_2 x_5 x_7 y_2y_3y_6y_7$} 
\end{minipage}} & 
\resizebox{.5\width}{.5\height}{
\begin{minipage}{0.4\linewidth}
\begin{tikzpicture}
\draw (0,0) rectangle (1,1);
\draw (1,0) rectangle (2,1);
\draw (2,0) rectangle(3,1);
\draw (2,1) rectangle(3,2);
\draw (3,1) rectangle(4,2);
\draw (4,1) rectangle(5,2);
\draw (4,2) rectangle(5,3);
\node at (0.5,0.5) {1};
\node at (1.5,0.5) {2};
\node at (2.5,0.5) {3};
\node at (2.5,1.5) {4};
\node at (3.5,1.5) {5};
\node at (4.5,1.5) {6};
\node at (4.5,2.5) {7};
\node[above] at (0.5,0.9) {2};
\node[above] at (1.5,0.9) {3};
\node[below] at (1.5,0.1) {1};
\node[below] at (2.5,0.1) {2};
\node[left] at (2.1,1.5) {3};
\node[right] at (2.9,0.5) {4};
\node[below] at (3.5,1.1) {4};
\node[above] at (2.5,1.9) {5};
\node[below] at (4.5,1.1) {5};
\node[above] at (3.5,1.9) {6};
\node[left] at (4.1,2.5) {6};
\node[right] at (4.9,1.5) {7};
\draw[red,ultra thick] (0,0) -- (1,0);
\draw[red,ultra thick] (0,1) -- (1,1);
\draw[red,ultra thick] (2,0) -- (3,0);
\draw[red,ultra thick] (2,1) -- (3,1);
\draw[red,ultra thick] (2,2) -- (3,2);
\draw[red,ultra thick] (4,1) -- (4,2);
\draw[red,ultra thick] (5,1) -- (5,2);
\draw[red,ultra thick] (4,3) -- (5,3);
\end{tikzpicture}
\caption*{$x^2_2 x_5 x_7 y_1y_2y_3y_6y_7$} 
\end{minipage}} & 
\resizebox{.5\width}{.5\height}{
\begin{minipage}{0.4\linewidth}
\begin{tikzpicture}
\draw (0,0) rectangle (1,1);
\draw (1,0) rectangle (2,1);
\draw (2,0) rectangle(3,1);
\draw (2,1) rectangle(3,2);
\draw (3,1) rectangle(4,2);
\draw (4,1) rectangle(5,2);
\draw (4,2) rectangle(5,3);
\node at (0.5,0.5) {1};
\node at (1.5,0.5) {2};
\node at (2.5,0.5) {3};
\node at (2.5,1.5) {4};
\node at (3.5,1.5) {5};
\node at (4.5,1.5) {6};
\node at (4.5,2.5) {7};
\node[above] at (0.5,0.9) {2};
\node[above] at (1.5,0.9) {3};
\node[below] at (1.5,0.1) {1};
\node[below] at (2.5,0.1) {2};
\node[left] at (2.1,1.5) {3};
\node[right] at (2.9,0.5) {4};
\node[below] at (3.5,1.1) {4};
\node[above] at (2.5,1.9) {5};
\node[below] at (4.5,1.1) {5};
\node[above] at (3.5,1.9) {6};
\node[left] at (4.1,2.5) {6};
\node[right] at (4.9,1.5) {7};
\draw[red,ultra thick] (0,0) -- (1,0);
\draw[red,ultra thick] (0,1) -- (1,1);
\draw[red,ultra thick] (2,0) -- (3,0);
\draw[red,ultra thick] (2,1) -- (2,2);
\draw[red,ultra thick] (3,1) -- (3,2);
\draw[red,ultra thick] (4,1) -- (5,1);
\draw[red,ultra thick] (4,2) -- (4,3);
\draw[red,ultra thick] (5,2) -- (5,3);
\end{tikzpicture}
\caption*{$x^2_2 x_3 x_5 x_6 y_1y_2y_3y_4$}
\end{minipage}} \\
\hline
\resizebox{.5\width}{.5\height}{
\begin{minipage}{0.4\linewidth}
\begin{tikzpicture}
\draw (0,0) rectangle (1,1);
\draw (1,0) rectangle (2,1);
\draw (2,0) rectangle(3,1);
\draw (2,1) rectangle(3,2);
\draw (3,1) rectangle(4,2);
\draw (4,1) rectangle(5,2);
\draw (4,2) rectangle(5,3);
\node at (0.5,0.5) {1};
\node at (1.5,0.5) {2};
\node at (2.5,0.5) {3};
\node at (2.5,1.5) {4};
\node at (3.5,1.5) {5};
\node at (4.5,1.5) {6};
\node at (4.5,2.5) {7};
\node[above] at (0.5,0.9) {2};
\node[above] at (1.5,0.9) {3};
\node[below] at (1.5,0.1) {1};
\node[below] at (2.5,0.1) {2};
\node[left] at (2.1,1.5) {3};
\node[right] at (2.9,0.5) {4};
\node[below] at (3.5,1.1) {4};
\node[above] at (2.5,1.9) {5};
\node[below] at (4.5,1.1) {5};
\node[above] at (3.5,1.9) {6};
\node[left] at (4.1,2.5) {6};
\node[right] at (4.9,1.5) {7};
\draw[red,ultra thick] (0,0) -- (1,0);
\draw[red,ultra thick] (0,1) -- (1,1);
\draw[red,ultra thick] (2,0) -- (3,0);
\draw[red,ultra thick] (2,1) -- (2,2);
\draw[red,ultra thick] (3,1) -- (3,2);
\draw[red,ultra thick] (4,1) -- (5,1);
\draw[red,ultra thick] (4,2) -- (5,2);
\draw[red,ultra thick] (4,3) -- (5,3);
\end{tikzpicture}
\caption*{$x^2_2 x_3 x_5 y_1y_2y_3y_4y_7$} 
\end{minipage}} &
\resizebox{.5\width}{.5\height}{
\begin{minipage}{0.4\linewidth}
\begin{tikzpicture}
\draw (0,0) rectangle (1,1);
\draw (1,0) rectangle (2,1);
\draw (2,0) rectangle(3,1);
\draw (2,1) rectangle(3,2);
\draw (3,1) rectangle(4,2);
\draw (4,1) rectangle(5,2);
\draw (4,2) rectangle(5,3);
\node at (0.5,0.5) {1};
\node at (1.5,0.5) {2};
\node at (2.5,0.5) {3};
\node at (2.5,1.5) {4};
\node at (3.5,1.5) {5};
\node at (4.5,1.5) {6};
\node at (4.5,2.5) {7};
\node[above] at (0.5,0.9) {2};
\node[above] at (1.5,0.9) {3};
\node[below] at (1.5,0.1) {1};
\node[below] at (2.5,0.1) {2};
\node[left] at (2.1,1.5) {3};
\node[right] at (2.9,0.5) {4};
\node[below] at (3.5,1.1) {4};
\node[above] at (2.5,1.9) {5};
\node[below] at (4.5,1.1) {5};
\node[above] at (3.5,1.9) {6};
\node[left] at (4.1,2.5) {6};
\node[right] at (4.9,1.5) {7};
\draw[red,ultra thick] (0,0) -- (1,0);
\draw[red,ultra thick] (0,1) -- (1,1);
\draw[red,ultra thick] (2,0) -- (3,0);
\draw[red,ultra thick] (2,1) -- (2,2);
\draw[red,ultra thick] (3,1) -- (3,2);
\draw[red,ultra thick] (4,1) -- (4,2);
\draw[red,ultra thick] (5,1) -- (5,2);
\draw[red,ultra thick] (4,3) -- (5,3);
\end{tikzpicture}
\caption*{$x^2_2 x_3 x_7 y_1y_2y_3y_4y_6y_7$} 
\end{minipage}} & 
\resizebox{.5\width}{.5\height}{
\begin{minipage}{0.4\linewidth}
\begin{tikzpicture}
\draw (0,0) rectangle (1,1);
\draw (1,0) rectangle (2,1);
\draw (2,0) rectangle(3,1);
\draw (2,1) rectangle(3,2);
\draw (3,1) rectangle(4,2);
\draw (4,1) rectangle(5,2);
\draw (4,2) rectangle(5,3);
\node at (0.5,0.5) {1};
\node at (1.5,0.5) {2};
\node at (2.5,0.5) {3};
\node at (2.5,1.5) {4};
\node at (3.5,1.5) {5};
\node at (4.5,1.5) {6};
\node at (4.5,2.5) {7};
\node[above] at (0.5,0.9) {2};
\node[above] at (1.5,0.9) {3};
\node[below] at (1.5,0.1) {1};
\node[below] at (2.5,0.1) {2};
\node[left] at (2.1,1.5) {3};
\node[right] at (2.9,0.5) {4};
\node[below] at (3.5,1.1) {4};
\node[above] at (2.5,1.9) {5};
\node[below] at (4.5,1.1) {5};
\node[above] at (3.5,1.9) {6};
\node[left] at (4.1,2.5) {6};
\node[right] at (4.9,1.5) {7};
\draw[red,ultra thick] (0,0) -- (0,1);
\draw[red,ultra thick] (1,0) -- (1,1);
\draw[red,ultra thick] (2,0) -- (3,0);
\draw[red,ultra thick] (2,1) -- (2,2);
\draw[red,ultra thick] (3,1) -- (3,2);
\draw[red,ultra thick] (4,1) -- (5,1);
\draw[red,ultra thick] (4,2) -- (4,3);
\draw[red,ultra thick] (5,2) -- (5,3);
\end{tikzpicture}
\caption*{$x_2 x_3 x_5 x_6 y_2y_3y_4$} 
\end{minipage}} & 
\resizebox{.5\width}{.5\height}{
\begin{minipage}{0.4\linewidth}
\begin{tikzpicture}
\draw (0,0) rectangle (1,1);
\draw (1,0) rectangle (2,1);
\draw (2,0) rectangle(3,1);
\draw (2,1) rectangle(3,2);
\draw (3,1) rectangle(4,2);
\draw (4,1) rectangle(5,2);
\draw (4,2) rectangle(5,3);
\node at (0.5,0.5) {1};
\node at (1.5,0.5) {2};
\node at (2.5,0.5) {3};
\node at (2.5,1.5) {4};
\node at (3.5,1.5) {5};
\node at (4.5,1.5) {6};
\node at (4.5,2.5) {7};
\node[above] at (0.5,0.9) {2};
\node[above] at (1.5,0.9) {3};
\node[below] at (1.5,0.1) {1};
\node[below] at (2.5,0.1) {2};
\node[left] at (2.1,1.5) {3};
\node[right] at (2.9,0.5) {4};
\node[below] at (3.5,1.1) {4};
\node[above] at (2.5,1.9) {5};
\node[below] at (4.5,1.1) {5};
\node[above] at (3.5,1.9) {6};
\node[left] at (4.1,2.5) {6};
\node[right] at (4.9,1.5) {7};
\draw[red,ultra thick] (0,0) -- (0,1);
\draw[red,ultra thick] (1,0) -- (1,1);
\draw[red,ultra thick] (2,0) -- (3,0);
\draw[red,ultra thick] (2,1) -- (2,2);
\draw[red,ultra thick] (3,1) -- (3,2);
\draw[red,ultra thick] (4,1) -- (5,1);
\draw[red,ultra thick] (4,2) -- (5,2);
\draw[red,ultra thick] (4,3) -- (5,3);
\end{tikzpicture}
\caption*{$x_2 x_3 x_5 y_2y_3y_4y_7$}
\end{minipage}} \\
\hline
\resizebox{.5\width}{.5\height}{
\begin{minipage}{0.4\linewidth}
\begin{tikzpicture}
\draw (0,0) rectangle (1,1);
\draw (1,0) rectangle (2,1);
\draw (2,0) rectangle(3,1);
\draw (2,1) rectangle(3,2);
\draw (3,1) rectangle(4,2);
\draw (4,1) rectangle(5,2);
\draw (4,2) rectangle(5,3);
\node at (0.5,0.5) {1};
\node at (1.5,0.5) {2};
\node at (2.5,0.5) {3};
\node at (2.5,1.5) {4};
\node at (3.5,1.5) {5};
\node at (4.5,1.5) {6};
\node at (4.5,2.5) {7};
\node[above] at (0.5,0.9) {2};
\node[above] at (1.5,0.9) {3};
\node[below] at (1.5,0.1) {1};
\node[below] at (2.5,0.1) {2};
\node[left] at (2.1,1.5) {3};
\node[right] at (2.9,0.5) {4};
\node[below] at (3.5,1.1) {4};
\node[above] at (2.5,1.9) {5};
\node[below] at (4.5,1.1) {5};
\node[above] at (3.5,1.9) {6};
\node[left] at (4.1,2.5) {6};
\node[right] at (4.9,1.5) {7};
\draw[red,ultra thick] (0,0) -- (0,1);
\draw[red,ultra thick] (1,0) -- (1,1);
\draw[red,ultra thick] (2,0) -- (3,0);
\draw[red,ultra thick] (2,1) -- (2,2);
\draw[red,ultra thick] (3,1) -- (3,2);
\draw[red,ultra thick] (4,1) -- (4,2);
\draw[red,ultra thick] (5,1) -- (5,2);
\draw[red,ultra thick] (4,3) -- (5,3);
\end{tikzpicture}
\caption*{$x_2 x_3 x_7 y_2y_3y_4y_6y_7$} 
\end{minipage}} & 
\resizebox{.5\width}{.5\height}{
\begin{minipage}{0.4\linewidth}
\begin{tikzpicture}
\draw (0,0) rectangle (1,1);
\draw (1,0) rectangle (2,1);
\draw (2,0) rectangle(3,1);
\draw (2,1) rectangle(3,2);
\draw (3,1) rectangle(4,2);
\draw (4,1) rectangle(5,2);
\draw (4,2) rectangle(5,3);
\node at (0.5,0.5) {1};
\node at (1.5,0.5) {2};
\node at (2.5,0.5) {3};
\node at (2.5,1.5) {4};
\node at (3.5,1.5) {5};
\node at (4.5,1.5) {6};
\node at (4.5,2.5) {7};
\node[above] at (0.5,0.9) {2};
\node[above] at (1.5,0.9) {3};
\node[below] at (1.5,0.1) {1};
\node[below] at (2.5,0.1) {2};
\node[left] at (2.1,1.5) {3};
\node[right] at (2.9,0.5) {4};
\node[below] at (3.5,1.1) {4};
\node[above] at (2.5,1.9) {5};
\node[below] at (4.5,1.1) {5};
\node[above] at (3.5,1.9) {6};
\node[left] at (4.1,2.5) {6};
\node[right] at (4.9,1.5) {7};
\draw[red,ultra thick] (0,0) -- (0,1);
\draw[red,ultra thick] (1,0) -- (1,1);
\draw[red,ultra thick] (2,0) -- (3,0);
\draw[red,ultra thick] (2,1) -- (2,2);
\draw[red,ultra thick] (3,1) -- (4,1);
\draw[red,ultra thick] (3,2) -- (4,2);
\draw[red,ultra thick] (5,1) -- (5,2);
\draw[red,ultra thick] (4,3) -- (5,3);
\end{tikzpicture}
\caption*{$x_2 x_3 x_4 x_6 x_7 y_2y_3y_4y_5y_6y_7$}
\end{minipage}} & 
\resizebox{.5\width}{.5\height}{
\begin{minipage}{0.4\linewidth}
\begin{tikzpicture}
\draw (0,0) rectangle (1,1);
\draw (1,0) rectangle (2,1);
\draw (2,0) rectangle(3,1);
\draw (2,1) rectangle(3,2);
\draw (3,1) rectangle(4,2);
\draw (4,1) rectangle(5,2);
\draw (4,2) rectangle(5,3);
\node at (0.5,0.5) {1};
\node at (1.5,0.5) {2};
\node at (2.5,0.5) {3};
\node at (2.5,1.5) {4};
\node at (3.5,1.5) {5};
\node at (4.5,1.5) {6};
\node at (4.5,2.5) {7};
\node[above] at (0.5,0.9) {2};
\node[above] at (1.5,0.9) {3};
\node[below] at (1.5,0.1) {1};
\node[below] at (2.5,0.1) {2};
\node[left] at (2.1,1.5) {3};
\node[right] at (2.9,0.5) {4};
\node[below] at (3.5,1.1) {4};
\node[above] at (2.5,1.9) {5};
\node[below] at (4.5,1.1) {5};
\node[above] at (3.5,1.9) {6};
\node[left] at (4.1,2.5) {6};
\node[right] at (4.9,1.5) {7};
\draw[red,ultra thick] (0,0) -- (1,0);
\draw[red,ultra thick] (0,1) -- (1,1);
\draw[red,ultra thick] (2,0) -- (3,0);
\draw[red,ultra thick] (2,1) -- (2,2);
\draw[red,ultra thick] (3,1) -- (4,1);
\draw[red,ultra thick] (3,2) -- (4,2);
\draw[red,ultra thick] (5,1) -- (5,2);
\draw[red,ultra thick] (4,3) -- (5,3);
\end{tikzpicture}
\caption*{$x^2_2 x_3 x_4 x_6 x_7 y_1y_2y_3y_4y_5y_6y_7$}
\end{minipage}} \\
\cline{1-3}
\end{tabular}
\caption{Perfect matchings of the labeled snake graph associated to $x[\alpha_{1,7}]$.} \label{perfect matchings x17}
\end{table}

By Theorem \ref{dual case of F-polynomial}, the associated $F$-polynomial is given by 
\[
F(y_1,y_2,y_3,y_4,y_5,y_6,y_7)=F(\mathcal{G}[2,3,3]) = C^{-1}_3\mathcal{N}[\mathcal{L}_1,\mathcal{L}_2,\mathcal{L}_3],
\] 
where
\begin{align*}
\mathcal{L}_1=1+y_1, \quad \mathcal{L}_2= y_2(1+y_3+y_3y_4),  \quad \mathcal{L}_3=y^{-1}_2y^{-1}_3y^{-1}_4y^{-1}_5y^{-1}_6y^{-1}_7(1+y_7+y_6y_7).
\end{align*}
By computation, we have 
\begin{multline*} 
F(\mathcal{G}[2,3,3])= (1+y_1)y_2(1+y_3+y_3y_4)(1+y_7+y_6y_7) + y_2y_3y_4y_5 y_6y_7(1+y_1)+ (1+y_7+y_6y_7).
\end{multline*} 
Using Lemma \ref{Lemma36}, we have
\begin{align*}
F|_{\mathbb{P}}(y_1,y_2,y_3,y_4,y_5,y_6,y_7)& =  (y_1y_2y_3y_4y_5y_6y_7)|_{\mathbb{P}} 
= \cfrac{1}{x'_3 x'_8}. 
\end{align*}

Therefore by Theorem \ref{snake graph formula}, 
\begin{align*}
& x[\alpha_{1,7}]  =  \frac{1}{x_1x_2x_3x_4x_5x_6x_7} \sum_{P\in \text{Match}(\mathcal{G})} \frac{ x(P) y(P) }{(\oplus_{P\in \text{Match}(\mathcal{G})} y(P))} \\
& = \frac{x_5x'_3 x'_8 }{x_2x_7} +  \frac{x_5 x_8 x'_3 x'_7 }{x_2x_6x_7}  +\frac{x_8 x'_3 x'_6}{x_2x_6}  + \frac{x_5 x'_2 x'_3 x'_8 }{x_1x_2x_3x_7} + \frac{x_5x_8 x'_2 x'_3 x'_7}{x_1x_2x_3x_6x_7} + \frac{x_8 x'_2 x'_3 x'_6}{x_1x_2x_3x_6} + \frac{x_5 x'_1 x'_3 x'_8}{x_1x_3x_7} \\
& + \frac{x_5 x'_2 x'_4 x'_8}{x_1x_3x_4x_7} + \frac{x_2 x_5 x'_1 x'_4 x'_8}{x_1x_3x_4x_7} + \frac{x_5 x_8 x'_1 x'_3 x'_7}{x_1x_3x_6x_7}  + \frac{x_5x_8 x'_2 x'_4 x'_7}{x_1x_3x_4x_6x_7} + \frac{x_2 x_5 x_8 x'_1 x'_4 x'_7}{x_1x_3x_4x_6x_7} + \frac{x_8 x'_1 x'_3 x'_6}{x_1x_3x_6}  \\
& + \frac{x_8 x'_2 x'_4 x'_6}{x_1x_3x_4x_6} + \frac{x_2 x_8 x'_1 x'_4 x'_6}{x_1x_3x_4x_6} + \frac{x_2 x'_1 x'_5 x'_8}{x_1x_4x_7}  +\frac{x_2 x_8 x'_1 x'_5 x'_7}{x_1x_4x_6x_7} + \frac{x_2 x_8 x'_1 x'_5 x'_6}{x_1x_4x_5x_6} + \frac{x'_2 x'_5 x'_8}{x_1x_4x_7}\\
&  + \frac{x_8 x'_2 x'_5 x'_7}{x_1x_4x_6x_7} + \frac{x_8 x'_2 x'_5 x'_6}{x_1x_4x_5x_6} + \frac{x_8 x'_2x'_6}{x_1x_5}  + \frac{x_2 x_8 x'_1 x'_6}{x_1x_5}.
\end{align*}

Replacing $x_\ell, x'_\ell$ by $q$-characters of the corresponding initial simple modules, we obtain the $q$-character of $L(Y_{1,-3}Y_{3,-7}Y_{6,-2}Y_{8,-6})$.
\end{example}

\subsection{Relation with Brito and Chari's results} \label{relation Brito-Chari result}
In \cite[Proposition in Section 2.5]{BC19}, Brito and Chari gave a non-recursive formula of $x[\alpha_{i,j}]$ using two sets  $\Gamma_{i,j}$ and $\Gamma'_{i,j}$. Recall that in \cite[Section 2.3]{BC19}, the set $\Gamma_{i,j}$ for $i,j\in I$  is defined as follows. $\Gamma_{i,j}=\{0\}$ if $j<i$ and if $i\leq j$ then $\Gamma_{i,j}$ consists of $(0,1)$ sequences $\varepsilon:=(\varepsilon_i,\ldots,\varepsilon_{j+1})$ with $(j-i+2)$ length subject to the following four conditions:
\begin{align}
& \varepsilon_i + \cdots +  \varepsilon_{i_\diamond} \leq 1 \leq \varepsilon_i + \cdots +  \varepsilon_{i_\diamond+1} \quad \text{$i_\diamond \leq j$}, \label{BC231} \\
& \varepsilon_{m+1} + \cdots +  \varepsilon_{(m+1)_\diamond} \leq 1 \leq \varepsilon_{m+1} + \cdots +  \varepsilon_{(m+1)_\diamond+1}  \quad \text{$i_\diamond \leq m=m_\diamond < j_\bullet$}, \label{BC232} \\
& \varepsilon_{\max\{i,j_\bullet+1\}} + \cdots +  \varepsilon_j  \leq 1, \label{BC233} \\
& \varepsilon_{j+1} = \begin{cases}
1  & \text{if $j=j_\diamond$}, \\
1- (\varepsilon_{\max\{i,j_\bullet+1\}} + \cdots +  \varepsilon_j ) & \text{otherwise}.
\end{cases} \label{BC234} 
\end{align} 
The set $\Gamma'_{i,j}$ is determined by $\Gamma_{i,j}$ consisting of $(-1,0,1)$ sequences $\varepsilon':=(\varepsilon'_i,\ldots,\varepsilon'_{j+1})$ \cite[Section 2.4]{BC19}: 
\begin{itemize}
\item[(1)] if $i_\bullet=m_\bullet$ or $\varepsilon_{\max\{i,(m_\bullet)_\bullet+1\}} + \cdots +  \varepsilon_{m_\bullet}=1$, then 
\[
\varepsilon'_m = \begin{cases}
(\delta_{m,m_\diamond}-1) \varepsilon_{m+1}- \delta_{m,m_\diamond} & \text{if $\varepsilon_{\max\{i,m_\bullet+1\}} + \cdots +  \varepsilon_m=0$}, \\
\delta_{m,m_\diamond}-(\varepsilon_m+\varepsilon_{m+1}) & \text{if $\varepsilon_{\max\{i,m_\bullet+1\}} + \cdots +  \varepsilon_m=1$},
\end{cases}
\]
\item[(2)]  if $m_\bullet\geq i$ and $\varepsilon_{\max\{i,(m_\bullet)_\bullet+1\}} + \cdots +  \varepsilon_{m_\bullet}=0$, then $\varepsilon'_m=\delta_{m,m_\diamond}(1-\varepsilon_{m+1})$, and 
\item[(3)]  if $j=j_\diamond$, then $\varepsilon'_{j+1} = 1- (\varepsilon_{\max\{i,j_\bullet+1\}} + \cdots +  \varepsilon_j )$, otherwise $\varepsilon'_{j+1}=\varepsilon_{\max\{i,j_\bullet+1\}} + \cdots +  \varepsilon_j$.
\end{itemize} 

Following \cite[Section 2.5]{BC19}, for any $\varepsilon\in \Gamma_{i,j}$ and the corresponding $\varepsilon' \in \Gamma'_{i,j}$, one defines
\begin{align*}
m^\varepsilon_{i,j} = x^{(1-\varepsilon_i)}_{i-1} x^{\varepsilon'_i}_i \cdots x^{\varepsilon'_j}_j x^{\varepsilon'_{j+1}}_{j+1}, \quad  f^\varepsilon_{i,j} = (x'_i)^{\varepsilon_i} \cdots (x'_j)^{\varepsilon_j}  (x'_{j+1})^{(1-\delta_{j,j_\diamond}) \varepsilon_{j+1}}.
\end{align*}

\begin{example}
For $x[\alpha_{1,7}]$ in  Example \ref{exmaple39}, the sets $\Gamma_{1,7}$ and $\Gamma'_{1,7}$, monomials $m^\varepsilon_{1,7}$ and $f^\varepsilon_{1,7}$ are listed in Table \ref{table:sequence}. The $(i+1)$-th row  ($1\leq i \leq 23$) in Table \ref{table:sequence} corresponds to the $i$-th perfect matching in Table \ref{perfect matchings x17} from up to down and from left to right. 

\begin{table}
\center
\begin{tabular}{|c|c|c|c|}
\hline
$\Gamma_{1,7}$  & $\Gamma'_{1,7}$ & $m^\varepsilon_{1,7}$  & $f^\varepsilon_{1,7}$ \\
\hline
 $(0,0,1,0,0,0,0,1)$ & $(0,-1,0,0,1,0,-1,0)$  & $x^{-1}_2x_5x^{-1}_7$   & $x'_3x'_8$ \\
\hline
$(0,0,1,0,0,0,1,0)$ & $(0,-1,0,0,1,-1,-1,1)$  &  $x^{-1}_2x_5x^{-1}_6x^{-1}_7x_8$  & $x'_3x'_7$  \\
\hline
$(0,0,1,0,0,1,0,0)$ & $(0,-1,0,0,0,-1,0,1)$  &  $x^{-1}_2x^{-1}_6x_8$  & $x'_3x'_6$  \\
\hline
$(0,1,1,0,0,0,0,1)$  &  $(-1,-1,-1,0,1,0,-1,0)$ & $x^{-1}_1x^{-1}_2x^{-1}_3x_5x^{-1}_7$  & $x'_2x'_3x'_8$  \\
\hline
$(0,1,1,0,0,0,1,0)$ & $(-1,-1,-1,0,1,-1,-1,1)$  &  $x^{-1}_1x^{-1}_2x^{-1}_3x_5x^{-1}_6x^{-1}_7x_8$  & $x'_2x'_3x'_7$  \\
\hline
$(0,1,1,0,0,1,0,0)$ & $(-1,-1,-1,0,0,-1,0,1)$   &  $x^{-1}_1x^{-1}_2x^{-1}_3x^{-1}_6x_8$  & $x'_2x'_3x'_6$  \\ 
\hline
$(1,0,1,0,0,0,0,1)$ & $(-1,0,-1,0,1,0,-1,0)$  & $x^{-1}_1x^{-1}_3x_5x^{-1}_7$  & $x'_1x'_3x'_8$  \\
\hline
 $(0,1,0,1,0,0,0,1)$ & $(-1,0,-1,-1,1,0,-1,0)$   &  $x^{-1}_1x^{-1}_3x^{-1}_4x_5x^{-1}_7$  & $x'_2x'_4x'_8$ \\ 
\hline
 $(1,0,0,1,0,0,0,1)$ & $(-1,1,-1,-1,1,0,-1,0)$   & $x^{-1}_1x_2x^{-1}_3x^{-1}_4x_5x^{-1}_7$  & $x'_1x'_4x'_8$  \\
\hline
$(1,0,1,0,0,0,1,0)$ & $(-1,0,-1,0,1,-1,-1,1)$  &  $x^{-1}_1x^{-1}_3x_5x^{-1}_6x^{-1}_7x_8$  & $x'_1x'_3x'_7$  \\
\hline
$(0,1,0,1,0,0,1,0)$ &  $(-1,0,-1,-1,1,-1,-1,1)$  &  $x^{-1}_1x^{-1}_3x^{-1}_4x_5x^{-1}_6x^{-1}_7x_8$  & $x'_2x'_4x'_7$  \\ 
\hline
$(1,0,0,1,0,0,1,0)$ & $(-1,1,-1,-1,1,-1,-1,1)$  &  $x^{-1}_1x_2x^{-1}_3x^{-1}_4x_5x^{-1}_6x^{-1}_7x_8$  & $x'_1x'_4x'_7$  \\ 
\hline
 $(1,0,1,0,0,1,0,0)$ & $(-1,0,-1,0,0,-1,0,1)$  &  $x^{-1}_1x^{-1}_3x^{-1}_6x_8$  & $x'_1x'_3x'_6$ \\
\hline 
$(0,1,0,1,0,1,0,0)$ & $(-1,0,-1,-1,0,-1,0,1)$  &  $x^{-1}_1x^{-1}_3x^{-1}_4x^{-1}_6x_8$  & $x'_2x'_4x'_6$  \\
\hline
$(1,0,0,1,0,1,0,0)$ & $(-1,1,-1,-1,0,-1,0,1)$  &  $x^{-1}_1x_2x^{-1}_3x^{-1}_4x^{-1}_6x_8$  & $x'_1x'_4x'_6$  \\
\hline
 $(1,0,0,0,1,0,0,1)$ & $(-1,1,0,-1,0,0,-1,0)$  &  $x^{-1}_1x_2x^{-1}_4x^{-1}_7$  & $x'_1x'_5x'_8$ \\ 
\hline
$(1,0,0,0,1,0,1,0)$ &  $(-1,1,0,-1,0,-1,-1,1)$  &  $x^{-1}_1x_2x^{-1}_4x^{-1}_6x^{-1}_7x_8$  & $x'_1x'_5x'_7$  \\ 
\hline
$(1,0,0,0,1,1,0,0)$ & $(-1,1,0,-1,-1,-1,0,1)$  &  $x^{-1}_1x_2x^{-1}_4x^{-1}_5x^{-1}_6x_8$  & $x'_1x'_5x'_6$  \\
\hline
 $(0,1,0,0,1,0,0,1)$ & $(-1,0,0,-1,0,0,-1,0)$  &  $x^{-1}_1x^{-1}_4x^{-1}_7$  & $x'_2x'_5x'_8$ \\ 
\hline
$(0,1,0,0,1,0,1,0)$ & $(-1,0,0,-1,0,-1,-1,1)$  &  $x^{-1}_1x^{-1}_4x^{-1}_6x^{-1}_7x_8$  & $x'_2x'_5x'_7$  \\ 
\hline
$(0,1,0,0,1,1,0,0)$ & $(-1,0,0,-1,-1,-1,0,1)$  &  $x^{-1}_1x^{-1}_4x^{-1}_5x^{-1}_6x_8$  & $x'_2x'_5x'_6$  \\
\hline
$(0,1,0,0,0,1,0,0)$ & $(-1,0,0,0,-1,0,0,1)$  &  $x^{-1}_1x^{-1}_5x_8$  & $x'_2x'_6$  \\
\hline
$(1,0,0,0,0,1,0,0)$ & $(-1,1,0,0,-1,0,0,1)$  &  $x^{-1}_1x_2x^{-1}_5x_8$  & $x'_1x'_6$  \\
\hline
\end{tabular}
\caption{The sets $\Gamma_{1,7}$ and $\Gamma'_{1,7}$, monomials $m^\varepsilon_{1,7}$ and $f^\varepsilon_{1,7}$.} \label{table:sequence}
\end{table}
\end{example}

In general, we have the following theorem.
\begin{theorem}\label{thm related to Brito and chari}
Let $\mathcal{G}:=\mathcal{G}[a_1,a_2,\ldots,a_n]$ be the snake graph associated to $x[\alpha_{i,j}]$ $(i\leq j)$. Then $|\Gamma_{i,j}|=|{\rm Match}(\mathcal{G})|=\mathcal{N}[a_1,a_2,\ldots,a_n]$.
\end{theorem}

\begin{proof}
By \cite[Lemma in Section 2.3]{BC19}, we have $|\Gamma_{i,j}|=|\Gamma_{i,j-1}| + |\Gamma_{i,j_\bullet-1}|$ for $j>i$ and $|\Gamma_{i,j}|=1$ for $j<i$. 

Recall that $\mathcal{G}$ has $\mathcal{N}[a_1,a_2,\ldots,a_n] $ many perfect matchings \cite[Theorem 3.4]{CS18} and 
\[
\mathcal{N}[a_1,a_2,\ldots,a_n] = \mathcal{N}[a_1,a_2,\ldots,a_n-1,1] = \mathcal{N}[a_1,a_2,\ldots,a_n-1] + \mathcal{N}[a_1,a_2,\ldots,a_{n-1}].
\]

In the following, we prove that $|\Gamma_{i,j}|=\mathcal{N}[a_1,a_2,\ldots,a_n]$ by an induction on $j-i$. If $j=i$, then $\mathcal{N}[2]=2$ and by the proof of \cite[Proposition in Section 2.5]{BC19}, we have 
\begin{align*}
\Gamma_{i,i} = \begin{cases} 
\{(0,1), (1,1)\} & \textit{ if $i=i_\diamond$}, \\
\{(0,1), (1,0)\} & \textit{ if $i\neq i_\diamond$}, 
\end{cases} \quad
\Gamma'_{i,i} = \begin{cases} 
\{(-1,1), (-1,0)\} & \textit{ if $i=i_\diamond$}, \\
\{(-1,0), (-1,1)\} & \textit{ if $i\neq i_\diamond$}.
\end{cases}
\end{align*}



Assume that our formula holds for $j-i\leq k$ and $|\Gamma_{i,i+k}|=\mathcal{N}[a_1,a_2,\ldots,a_n]$, where $\mathcal{G}[a_1,a_2,\ldots,a_n]$ is the snake graph associated to $x[\alpha_{i,i+k}]$. If $i+k\neq (i+k)_\diamond$, then 
\begin{align*}
|\Gamma_{i,i+k+1}| & = |\Gamma_{i,i+k}| + |\Gamma_{i,(i+k+1)_\diamond-1}|= \mathcal{N}[a_1,a_2,\ldots,a_n] + \mathcal{N}[a_1,a_2,\ldots,a_{n-1}] \\
& = \mathcal{N}[a_1,a_2,\ldots,a_n+1].
\end{align*}
In this case, $\mathcal{G}[a_1,a_2,\ldots,a_n+1]$ is the snake graph associated to $x[\alpha_{i,i+k+1}]$. 

If $i+k=(i+k)_\diamond$, then 
\begin{align*}
|\Gamma_{i,i+k+1}| & = |\Gamma_{i,i+k}| + |\Gamma_{i,(i+k+1)_\diamond-1}| = \mathcal{N}[a_1,a_2,\ldots,a_n] +  |\Gamma_{i, i+k-1}|  \\
& = \mathcal{N}[a_1,a_2,\ldots,a_n-1,1]+\mathcal{N}[a_1,a_2,\ldots,a_n-1] \\
& = \mathcal{N}[a_1,a_2,\ldots,a_n-1,2].
\end{align*}
In this case, $\mathcal{G}[a_1,a_2,\ldots,a_n-1,2]$ is the snake graph associated to $x[\alpha_{i,i+k+1}]$.
\end{proof}

We refine Theorem \ref{thm related to Brito and chari} as follows.
\begin{theorem}
Let $\mathcal{G}=(G_1,\ldots,G_{j-i+1})$ be the labeled snake graph associated to $x[\alpha_{i,j}]$. Then

\begin{align}
|\{(\varepsilon_i,\ldots,\varepsilon_{i_\diamond+1}) \mid (\varepsilon_i,\ldots,\varepsilon_{i_\diamond+1}, \ldots) \in \Gamma_{i,j}\}|=|{\rm Match}(G_1,\ldots,G_{i_\diamond-i+2})|,  \label{refined equation1}
\end{align} 
\begin{align}
\begin{split}
& |\{(\varepsilon_{m+1}, \ldots, \varepsilon_{(m+1)_\diamond+1}) \mid (\ldots, \varepsilon_{m+1}, \ldots, \varepsilon_{(m+1)_\diamond+1}, \ldots) \in  \Gamma_{i,j}\} | \\
& = |{\rm Match}(G_{m-i+2},\ldots,G_{(m+1)_\diamond-i+2})|,  \label{refined equation2}
\end{split}
\end{align} 
\begin{align}
\begin{split}
& |\{(\varepsilon_{\max\{i,j_\bullet+1\}}, \ldots, \varepsilon_{j+1}) \mid (\ldots, \varepsilon_{\max\{i,j_\bullet+1\}}, \ldots, \varepsilon_{j+1}) \in \Gamma_{i,j}\}| \\
& =|{\rm Match}(G_{\max\{i,j_\bullet+1\}-i+1} ,\ldots,G_{j-i+1})|. \label{refined equation3}
\end{split}
\end{align}
\end{theorem} 

\begin{proof}
Indeed, by (\ref{BC231}), the set in the left hand side of (\ref{refined equation1}) is equivalent to 
\[
\{(\varepsilon_i,\ldots,\varepsilon_{i_\diamond+1})\in \mathbb{Z}_{\geq 0}^{(i_\diamond-i+2)} \mid (\varepsilon_i + \cdots +  \varepsilon_{i_\diamond}, \varepsilon_{i_\diamond+1})\in \{(0,1), (1,0), (1,1)\}\},
\]
the cardinality of the set is $2(i_\diamond-i)+3$; by (\ref{BC232}), the set in the left hand side of (\ref{refined equation2}) is equivalent to 
\begin{align*}
\left\{(\varepsilon_{m+1}, \ldots, \varepsilon_{(m+1)_\diamond+1})\in \mathbb{Z}_{\geq 0}^{((m+1)_\diamond-m+1)} \mid (\varepsilon_{m+1}+\cdots+\varepsilon_{(m+1)_\diamond}, \varepsilon_{(m+1)_\diamond+1}) \right. \\
\left. \in \{(0,1), (1,0), (1,1)\}  \right\},
\end{align*}
the cardinality of the set is $2((m+1)_\diamond-m)+1$; by (\ref{BC233})--(\ref{BC234}), the set in the left hand side of (\ref{refined equation3}) is equivalent to 
\begin{align*}
\begin{cases} 
\left\{(\varepsilon_{\max\{i,j_\bullet+1\}}, \ldots, \varepsilon_{j+1})  \in \mathbb{Z}_{\geq 0}^{(j-\max\{i,j_\bullet+1\}+2)} \mid (\varepsilon_{\max\{i,j_\bullet+1\}}, \ldots, \varepsilon_j, \varepsilon_{j+1})  \right. \\
\left. \in  \{(0,1), (1,1)\} \right\},  & \text{ if $j=j_\diamond$}, \\
\left\{(\varepsilon_{\max\{i,j_\bullet+1\}}, \ldots, \varepsilon_{j+1})  \in \mathbb{Z}_{\geq 0}^{(j-\max\{i,j_\bullet+1\}+2)} \mid (\varepsilon_{\max\{i,j_\bullet+1\}}, \ldots, \varepsilon_j, \varepsilon_{j+1})  \right. \\
\left. \in  \{(0,1), (1,0) \right\}, & \text{ if $j\neq j_\diamond$}, \\
\end{cases}
\end{align*}
the cardinality of the set is $(j- \max\{i,j_\bullet+1\}+2)$.

In the other hand, the snake graph in (\ref{refined equation1}) or (\ref{refined equation2}) is a union of a zigzag snake graph and a tile, whereas the snake graph in (\ref{refined equation3}) is a zigzag snake graph. A zigzag snake graph with $d$ tiles has $(d+1)$ many perfect matchings. Finally, we apply the following formula
\[
\mathcal{N}[a,2] = 2a+1,
\]
where $a$ is one less than the number of tiles in a snake graph.
\end{proof}
 
%
%
%

\begin{remark}\label{conjecture313}
Let $\mathcal{G}$ be the labeled snake graph associated to $x[\alpha_{i,j}]$ and $\Gamma_{i,j}$ the set defined in the beginning of this section. There is a partial order on the set ${\rm Match}(\mathcal{G})$ \cite{MSW11}. We expect that the following are true and they are equivalent.
\begin{itemize}
\item[(1)] There exsits a partial order on $\Gamma_{i,j}$ such that there is a bijection from $\Gamma_{i,j}$ to ${\rm Match}(\mathcal{G})$ preserving the order.
\item[(2)] For any $\varepsilon \in \Gamma_{i,j}$, there exists a unique perfect matching $P^\varepsilon\in {\rm Match}(\mathcal{G})$ such that
\begin{align*}
m^\varepsilon_{i,j} f^\varepsilon_{i,j} = \cfrac{x(P^\varepsilon)y(P^\varepsilon)}{(\prod_{\ell=i}^j x_\ell) y(P_+)|_{\mathbb{P}}}. 
\end{align*} 
\end{itemize}
\end{remark}

\subsection{The highest and lowest $\ell$-weight monomials of Hernandez-Leclerc modules}
In this subsection, we determine the highest and lowest $\ell$-weight monomials in the $q$-character of an HL-module using perfect matchings of snake graphs. 

\begin{lemma}\label{the lemma310}
Let $Q_\xi$ be a quiver associated to a height function $\xi$. Assume that there is no source or sink vertex in the interval $(i,j)$. Then there exists a unique perfect matching $P$ such that
\[
\cfrac{x(P)y(P)}{{(x_i x_{i+1} \ldots x_j) F|_{\mathbb{P}}(y_i,y_{i+1},\ldots,y_j)}} = \cfrac{x'_i x^{(1-\delta_{j,j_\diamond})}_{j+1}}{x_i}  \ {\rm or } \ \cfrac{x'_i x'_{i+1} x^{(1-\delta_{j,j_\diamond})}_{j+1}}{x_ix_{i+1}}.
\]
\end{lemma}

\begin{proof}
If $\xi(i) = \xi(i+1)+1$, then there is a full subquiver of $Q_\xi$, which is one of the following quivers

\begin{align}
\xymatrix{
(i-1) &  i \ar[l] \ar[dr]  & (i+1) \ar[l]  \ar[dr]  &  \ldots \ar[l] \ar[dr]  & (j-1) \ar[l] \ar[dr]  & j \ar[l] \ar@/^10pt/[rr]^{\delta_{j,j_\diamond}}  \ar[drr]_{1-\delta_{j,j_\diamond}}  &&  (j+1) \ar@/^10pt/[ll]_{1-\delta_{j,j_\diamond}}  \\
 & i' \ar[u] & (i+1)'  \ar[u] & \ldots \ar[u]& (j-1)'  \ar[u]  & j' \ar[u]  && (j+1)'},  \label{case 34} \\
\xymatrix{
(i-1)  &  i \ar[l] \ar[r]  & (i+1) \ar[r] \ar[d]   & \ldots \ar[r] \ar[d]  & (j-1) \ar[r] \ar[d]  &  j \ar[d] \ar@/_10pt/[rr]^{1-\delta_{j,j_\diamond}}  &&  (j+1) \ar@/_10pt/[ll]_{\delta_{j,j_\diamond}}   \\
 & i' \ar[u] & (i+1)'  & \ldots \ar[ul] & (j-1)'  \ar[ul] & j' \ar[ul]  && (j+1)' \ar[ull]^{1-\delta_{j,j_\diamond}}}.  \label{case 35}
\end{align}

If $\xi(i) = \xi(i+1)-1$, then there is a full subquiver of $Q_\xi$, which is one of the following quivers

\begin{align}
\xymatrix{
(i-1) \ar[r]  &  i \ar[r] \ar[d]  & (i+1) \ar[r] \ar[d]   & \ldots \ar[r] \ar[d]  & (j-1) \ar[r] \ar[d]  &  j \ar[d] \ar@/_10pt/[rr]^{1-\delta_{j,j_\diamond}}  &&  (j+1) \ar@/_10pt/[ll]_{\delta_{j,j_\diamond}}   \\
 & i' & (i+1)'  \ar[ul]  & \ldots \ar[ul] & (j-1)'  \ar[ul] & j' \ar[ul]  && (j+1)' \ar[ull]^{1-\delta_{j,j_\diamond}}},  \label{case 36} \\
\xymatrix{
(i-1) \ar[r]  &  i \ar[d]  & (i+1) \ar[l]  \ar[dr]  &  \ldots \ar[l] \ar[dr]  & (j-1) \ar[l] \ar[dr]  & j \ar[l] \ar@/^10pt/[rr]^{\delta_{j,j_\diamond}}  \ar[drr]_{1-\delta_{j,j_\diamond}}  &&  (j+1) \ar@/^10pt/[ll]_{1-\delta_{j,j_\diamond}}  \\
 & i' & (i+1)'  \ar[u] & \ldots \ar[u]& (j-1)'  \ar[u]  & j' \ar[u]  && (j+1)'}.  \label{case 37}
\end{align}

By our Definition \ref{HL and snake graph}, the labeled snake graph associated to $\iota(x[\alpha_{i,j}])$ is a zigzag labeled snake graph $\mathcal{G}=(G_i,G_{i+1},\ldots,G_j)$. Let $P_{-}$ (respectively, $P_{+}$) be the minimal (respectively, maximal) perfect matching of $\mathcal{G}$.  Define another labeled perfect matching $P$ of $\mathcal{G}$ as follows:
\begin{align*}
P = \begin{cases}
P_{+} & \text{for (\ref{case 34})}, \\
_W\mathcal{G} \cup e_1 \cup P_{-}|_{(\mathcal{G}\backslash G_i)} & \text{for (\ref{case 35})}, \\
P_{-} & \text{for (\ref{case 36})}, \\
_W\mathcal{G}  \cup e_1 \cup P_{+}|_{(\mathcal{G}\backslash G_i)}  & \text{for (\ref{case 37})}.
\end{cases}
\end{align*}
Then by Lemma \ref{Lemma36}, we have

\begin{align*}
F|_{\mathbb{P}}(y_i,y_{i+1},\ldots,y_j)  = \begin{cases}
\cfrac{1}{x_{i-1} x^{\delta_{j,j_\diamond}}_{j+1} {x'}^{(1-\delta_{j,j_\diamond})}_{j+1}}   &  \text{for (\ref{case 34})}, \\
\cfrac{1}{x_{i-1} x'_{i+1} x^{(1-\delta_{j,j_\diamond})}_{j+1}}   &  \text{for (\ref{case 35})},  \\
\cfrac{1}{x'_i x^{(1-\delta_{j,j_\diamond})}_{j+1}}  & \text{for (\ref{case 36})}, \\ 
\cfrac{1}{x'_i x^{\delta_{j,j_\diamond}}_{j+1} {x'}^{(1-\delta_{j,j_\diamond})}_{j+1}}   & \text{for (\ref{case 37})}.
\end{cases}
\end{align*}

By computation, we have 

\begin{align*}
\cfrac{x(P)}{x_i x_{i+1} \ldots x_j} & = \begin{cases}
\cfrac{1}{x_i} &  \text{for (\ref{case 34}) and (\ref{case 36})}, \\
\cfrac{1}{x_i x_{i+1}} & \text{for (\ref{case 35}) and (\ref{case 37})},
\end{cases} \\
\cfrac{y(P)}{F|_{\mathbb{P}}(y_i,y_{i+1},\ldots,y_j)} &  = \begin{cases}
x'_i x^{(1-\delta_{j,j_\diamond})}_{j+1}  & \text{for (\ref{case 34}) and (\ref{case 36})}, \\
x'_i x'_{i+1} x^{(1-\delta_{j,j_\diamond})}_{j+1}  & \text{for (\ref{case 35}) and (\ref{case 37})}.
\end{cases}
\end{align*}
Therefore

\begin{align*}
\cfrac{x(P) y(P)}{(x_i x_{i+1} \ldots x_j) F|_{\mathbb{P}}(y_i,y_{i+1},\ldots,y_j)}= 
\begin{cases}
\cfrac{x'_i x^{(1-\delta_{j,j_\diamond})}_{j+1}}{x_i} & \text{for (\ref{case 34}), (\ref{case 36})}, \\ 
\cfrac{x'_i x'_{i+1} x^{(1-\delta_{j,j_\diamond})}_{j+1}}{x_ix_{i+1}} & \text{for (\ref{case 35}), (\ref{case 37})}.
\end{cases}
\end{align*}

Finally, the uniqueness of $P$ follows from the uniqueness of the perfect matching with the height monomial $y(P)$, because the set of all perfect matchings in a zigzag snake graph forms a total order, and a weight monomial does not affect the associated height monomial.
\end{proof}

Denote by $\mathcal{G}=(G_i,G_{i+1},\ldots,G_j)$ the snake graph with sign function $(a_1,a_2,\ldots,a_l)$ associated to the HL-module parameterized by $\alpha_{i,j}$. We have the following observation: The vertices $(i+\ell_1-1), (i+\ell_2-1), \ldots, (i+\ell_{l-1}-1)$ are sources or sinks in $Q_\xi$. Hence by Definition \ref{HL and snake graph}, these tiles $G_{i+\ell_1-1}, G_{i+\ell_2-1},\ldots, G_{i+\ell_{l-1}-1}$ are sign-changed tiles.

Let 
\[
\mathcal{H}_t = (G_{i+\ell_{t-1}}, G_{i+\ell_{t-1}+1}, \ldots, G_{i+\ell_t-2}), \text{  $1 \leq t \leq l$}, 
\]
be all the (labeled) subsnake graphs in the decomposition of $\mathcal{G}$. Let $P_{-}$ (respectively, $P_{+}$) be the minimal (respectively, maximal) perfect matching of $\mathcal{G}$. Denote by $P_{-}|_{\mathcal{H}_t}$ (respectively, $P_{+}|_{\mathcal{H}_t}$) the minimal (respectively, maximal) perfect matching obtained by restricting  $P_{-}$ (respectively, $P_{+}$)  to $\mathcal{H}_t$, where  $1\leq t \leq l$.  Define another (labeled) perfect matching $P$ of $\mathcal{G}$ as follows. If $i$ is a source or sink in $Q_\xi$, we define
\begin{align*}
P|_{\mathcal{H}_1}= \begin{cases}
_W\mathcal{G} \cup e_1 \cup P_{-}|_{(\mathcal{H}_1\backslash G_i)}  &\text{if $i$ is a source}, \\
_W\mathcal{G}  \cup  e_1 \cup P_{+}|_{(\mathcal{H}_1\backslash G_i)} & \text{if $i$ is a sink},
\end{cases}
\end{align*}
and otherwise let
\begin{align*}
P|_{\mathcal{H}_1}= \begin{cases}
P_{-}|_{\mathcal{H}_1} &\text{if $i \to i'$ in $Q_\xi$},  \\
P_{+}|_{\mathcal{H}_1} & \text{if $i' \to i$ in $Q_\xi$}.
\end{cases}
\end{align*}

For $1< t \leq l$, let
\begin{align*}
P|_{\mathcal{H}_t}= \begin{cases}
P_{-}|_{\mathcal{H}_t} &  \text{if $(i+\ell_{t-1}) \to (i+\ell_{t-1})'$ in $Q_\xi$},  \\
P_{+}|_{\mathcal{H}_t} & \text{if $(i+\ell_{t-1})' \to (i+\ell_{t-1})$ in $Q_\xi$}.
\end{cases}
\end{align*} 
Define $P$ as the gluing of $P|_{\mathcal{H}_t}$  for $1\leq t \leq l$ and it is a perfect matching of $\mathcal{G}$.

\begin{definition}\label{definition311}
If $i$ is a source or sink in $Q_\xi$, then we define the revised height monomial 
\begin{align*}
\widetilde{y}(P|_{\mathcal{H}_1})=\begin{cases}
y_i  &  \text{if $i$ is a source}, \\
y(P|_{\mathcal{H}_1})   & \text{if $i$ is a sink and $l=1$}, \\
y_{i+\ell_1-1} y(P|_{\mathcal{H}_1})   & \text{if $i$ is a sink and $l>1$}. 
\end{cases}
\end{align*}
Otherwise, for $1\leq t \leq l$ we define the revised height monomial 
\begin{align*}
\widetilde{y}(P|_{\mathcal{H}_t})=\begin{cases}
1  &  \text{if $(i+\ell_{t-1}) \to (i+\ell_{t-1})'$ in $Q_\xi$},  \\  
y(P|_{\mathcal{H}_t})  &   \text{if $(i+\ell_{t-1})' \to (i+\ell_{t-1})$ in $Q_\xi$ and $t=l$}, \\
y_{i+\ell_t-1} y(P|_{\mathcal{H}_t})   &  \text{if $(i+\ell_{t-1})' \to (i+\ell_{t-1})$ in $Q_\xi$  and $t\neq l$}.
\end{cases}
\end{align*}
\end{definition}

We immediately have the following lemma.

\begin{lemma}\label{lemma312}
With the notation above, there exists a unique perfect matching $P\in {\rm Match}(\mathcal{G})$ such that
\begin{align*}
\cfrac{x(P) y(P)}{(\prod\limits_{s=i}^j x_s) F|_{\mathbb{P}}(y_i,y_{i+1},\ldots,y_j)} = \begin{cases}
\cfrac{x'_i x'_{i+1} x'_{i+\ell_1} x'_{i+\ell_2}\ldots x'_{i+\ell_{l-1}} x_{j+1}^{(1-\delta_{j,j_\diamond})}}{x_i x_{i+1} x_{i+\ell_1}x_{i+\ell_2}\ldots x_{i+\ell_{l-1}}}  & \text{if $i$ is a source or sink}, \\   \\
\cfrac{x'_i x'_{i+\ell_1} x'_{i+\ell_2}\ldots x'_{i+\ell_{l-1}} x_{j+1}^{(1-\delta_{j,j_\diamond})}}{x_i x_{i+\ell_1}x_{i+\ell_2}\ldots x_{i+\ell_{l-1}}}  & \text{otherwise}. 
\end{cases}
\end{align*}
\end{lemma}

\begin{proof} 
By the definition of $P$, if $i$ is a source or sink in $Q_\xi$, then we have 
\begin{align*}
& x(P|_{\mathcal{H}_1})=\begin{cases}
\prod\limits_{s=2}^{\ell_1-2} x_{i+s} & \text{if $l=1$}, \\
\prod\limits_{s=2}^{\ell_1-1} x_{i+s} & \text{if $l>1$}, 
\end{cases} \\ 
& \widetilde{y}(P|_{\mathcal{H}_1})=\begin{cases}
y_i  &  \text{if $i$ is a source}, \\
\prod\limits_{s=1}^{\ell_1-2}  y_{i+s}  & \text{if $i$ is a sink and $l=1$}, \\
\prod\limits_{s=1}^{\ell_1-1}  y_{i+s} & \text{if $i$ is a sink and $l>1$},
\end{cases}
\end{align*}
and otherwise for $1\leq t \leq l$
\begin{align*}
& x(P|_{\mathcal{H}_t})= \begin{cases}
\prod\limits_{s=\ell_{t-1}+1}^{\ell_t-2} x_{i+s}  & \text{if $t = l$}, \\
\prod\limits_{s=\ell_{t-1}+1}^{\ell_t-1} x_{i+s}   & \text{if $t < l$},
\end{cases} \\
& \widetilde{y}(P|_{\mathcal{H}_t})=\begin{cases}
1  &  \text{if $(i+\ell_{t-1}) \to (i+\ell_{t-1})'$ in $Q_\xi$},  \\  
\prod\limits_{s=\ell_{t-1}}^{\ell_t-2} y_{i+s}  &   \text{if $(i+\ell_{t-1})' \to (i+\ell_{t-1})$ in $Q_\xi$ and $t=l$}, \\
\prod\limits_{s=\ell_{t-1}}^{\ell_t-1} y_{i+s}   &  \text{if $(i+\ell_{t-1})' \to (i+\ell_{t-1})$ in $Q_\xi$  and $t\neq l$}.
\end{cases}
\end{align*}

By Lemma \ref{Lemma36}, we have 
\begin{align*} 
F|_{\mathbb{P}}(y_i,y_{i+1},\ldots,y_j)=y(P_{+})|_{\mathbb{P}} = \left(\prod_{t=1}^{l-1} (\prod_{s=\ell_{t-1}}^{\ell_t-1} y_{i+s})|_{\mathbb{P}}\right) (\prod_{s=\ell_{l-1}}^{\ell_l-2} y_{i+s})|_{\mathbb{P}}.
\end{align*}

We consider zigzag snake graphs 
\[
(G_{i+\ell_{t-1}}, G_{i+\ell_{t-1}+1}, \ldots, G_{i+\ell_t-1})=\mathcal{H}_t \cup G_{i+\ell_t-1}
\]
for $1\leq t \leq l-1$ and 
\[
(G_{i+\ell_{l-1}}, G_{i+\ell_{l-1}+1}, \ldots, G_{i+\ell_l-2})=\mathcal{H}_l.
\] 
By checking case by case, for $1\leq t \leq l-1$, $x(P|_{\mathcal{H}_t})$ is just the weight monomial on $\mathcal{H}_t \cup G_{i+\ell_t-1}$ with the height monomial $\widetilde{y}(P|_{\mathcal{H}_j})$, and $x(P|_{\mathcal{H}_l})$ is the weight monomial on $\mathcal{H}_l$ with the height monomial $\widetilde{y}(P|_{\mathcal{H}_l})$.

Since these vertices $(i+\ell_1-1), (i+\ell_2-1), \ldots, (i+\ell_{l-1}-1)$ are sources or sinks in $Q_\xi$, by Lemma \ref{the lemma310}, we have

\begin{multline*}
\cfrac{x(P) y(P)}{(\prod\limits_{s=1}^j x_s) F|_{\mathbb{P}}(y_i,y_{i+1},\ldots,y_j)} = \left(\prod_{t=1}^{l-1} \cfrac{x(P|_{\mathcal{H}_t}) \widetilde{y}(P|_{\mathcal{H}_t})}{(\prod\limits_{s=\ell_{t-1}}^{\ell_t-1} x_{i+s})(\prod\limits_{s=\ell_{t-1}}^{\ell_t-1} y_{i+s})|_{\mathbb{P}}} \right) \cfrac{x(P|_{\mathcal{H}_l}) \widetilde{y}(P|_{\mathcal{H}_l})}{(\prod\limits_{s=\ell_{l-1}}^{\ell_l-2} x_{i+s})(\prod\limits_{s=\ell_{l-1}}^{\ell_l-2} y_{i+s})|_{\mathbb{P}}} \\
= \begin{cases}
\cfrac{x'_i x'_{i+1}}{x_i x_{i+1}}  \left(\prod\limits_{t=2}^{l-1} \cfrac{x'_{i+\ell_{t-1}}}{x_{i+\ell_{t-1}}} \right) \cfrac{x'_{i+\ell_{l-1}} x_{j+1}^{(1-\delta_{j,j_\diamond})}}{x_{i+\ell_{l-1}}} & \text{if $i$ is a source or sink}, \\  
\left(\prod\limits_{j=1}^{l-1} \cfrac{x'_{i+\ell_{j-1}}}{x_{i+\ell_{j-1}}} \right) \cfrac{x'_{i+\ell_{l-1}} x_{j+1}^{(1-\delta_{j,j_\diamond})}}{x_{i+\ell_{l-1}}}  & \text{otherwise}.
\end{cases}
\end{multline*}

Finally, the uniqueness of $P$ follows from the uniqueness of the perfect matching of each $\mathcal{H}_t$ for $1\leq t \leq l$.
\end{proof}

By Theorem \ref{BCTheorem1} and Theorem \ref{HL defintion equivalent}, assume without loss of generality that  
\[
\iota(x[\alpha_{i,j}])=[L(Y_{i_1,a_1}Y_{i_2,a_2}\ldots Y_{i_k,a_k})]
\] 
for $i=i_1 < i_2 < \cdots < i_k$. Denote by $\mathcal{G}=(G_i,G_{i+1},\ldots,G_j)$ the snake graph with sign function $(a_1,a_2,\ldots,a_l)$ associated to the HL-module $L(Y_{i_1,a_1}Y_{i_2,a_2}\ldots Y_{i_k,a_k})$. Then by Lemma \ref{lemma312}, either $l=k$ or $|l-k|=1$, and $l=k$ if and only if $i$ is not a source or sink and $j$ is a source or sink.

\begin{theorem}\label{theorem313}
With the notation above, the highest or lowest $\ell$-weight monomial in the $q$-character of an arbitrary HL-module $L(Y_{i_1,a_1}Y_{i_2,a_2}\ldots Y_{i_k,a_k})$ occurs in
\begin{align*}
\begin{cases}
\cfrac{\chi_q(\iota(x'_i x'_{i+1} x'_{i+\ell_1} x'_{i+\ell_2}\ldots x'_{i+\ell_{l-1}} x_{j+1}^{(1-\delta_{j,j_\diamond})}))}{\chi_q(\iota(x_i x_{i+1} x_{i+\ell_1}x_{i+\ell_2}\ldots x_{i+\ell_{l-1}}))}  & \text{if $i$ is a source or sink}, \\   \\
\cfrac{\chi_q(\iota(x'_i x'_{i+\ell_1} x'_{i+\ell_2}\ldots x'_{i+\ell_{l-1}} x_{j+1}^{(1-\delta_{j,j_\diamond})}))}{\chi_q(\iota(x_i x_{i+\ell_1}x_{i+\ell_2}\ldots x_{i+\ell_{l-1}}))}  & \text{otherwise}. 
\end{cases}
\end{align*}
\end{theorem}

\begin{proof}
By Theorem \ref{snake graph formula}, we have 
\begin{align}\label{weight monomial}
\chi_q(\prod_{\ell=i}^j \iota(x_\ell))  \chi_q(\iota(x[\alpha_{i,j}])) = \chi_q\left(\iota(\frac{\sum_{P\in {\rm Match}(\mathcal{G})}  x(P) y(P) }{\bigoplus_{P\in {\rm Match}(\mathcal{G})} y(P)})\right).
\end{align}

By \cite[Corollary 6.9]{FM01}, for an arbitrary simple $U_q(\widehat{\mathfrak{g}})$-module $L(m)$, the lowest $\ell$-weight monomial in $\chi_q(L(m))$ is the product of the lowest $\ell$-weight monomials of fundamental modules whose highest $\ell$-weight monomials are factors of $m$, and the highest or lowest $\ell$-weight monomial is unique. Hence the lowest $\ell$-weight monomial in both sides of Equation (\ref{weight monomial}) should be same. By Lemma \ref{lemma312}, there exists a unique perfect matching $P$ of the snake graph associated to the HL-module $L(Y_{i_1,a_1}Y_{i_2,a_2}\ldots Y_{i_k,a_k})$ such that  
\begin{align*}
\cfrac{x(P) y(P)}{F|_{\mathbb{P}}(y_i,y_{i+1},\ldots,y_j)} = \begin{cases}
\cfrac{x'_i x'_{i+1} x'_{i+\ell_1} x'_{i+\ell_2}\ldots x'_{i+\ell_{l-1}} x_{j+1}^{(1-\delta_{j,j_\diamond})}}{x_i x_{i+1} x_{i+\ell_1}x_{i+\ell_2}\ldots x_{i+\ell_{l-1}}} (\prod\limits_{s=i}^j x_s)  & \text{if $i$ is a source or sink}, \\   \\
\cfrac{x'_i x'_{i+\ell_1} x'_{i+\ell_2}\ldots x'_{i+\ell_{l-1}} x_{j+1}^{(1-\delta_{j,j_\diamond})}}{x_i x_{i+\ell_1}x_{i+\ell_2}\ldots x_{i+\ell_{l-1}}} (\prod\limits_{s=i}^j x_s)  & \text{otherwise}. 
\end{cases}
\end{align*}

By Theorem \ref{BCTheorem1}, for any $1\leq \ell \leq k$, 
\[
\iota(x_{i_\ell}) = [L(Y_{i_\ell,\xi(i_\ell+1)})], \quad \iota(x'_{i_\ell}) = [L(Y_{i_\ell,\xi(i_\ell)-1}Y_{i_\ell,\xi(i_\ell)+1})].
\]
Comparing the highest $\ell$-weight monomial and the lowest $\ell$-weight monomial in both sides of Equation (\ref{weight monomial}), and by their uniqueness property, our result holds.
\end{proof}

\begin{example}
Continue our previous Example \ref{exmaple39}, the minimal perfect matching $P_{-}$ of $\mathcal{G}$ and each $P_{-}|_{\mathcal{H}_t}$ ($t=1,2,3$) are shown in the left and right hand side of the following figure respectively. 
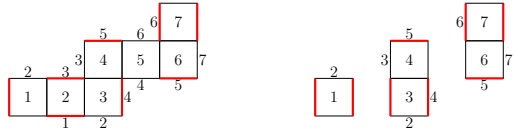
\begin{figure}[H]
\resizebox{.5\width}{.5\height}{
\begin{minipage}{0.5\linewidth}
\begin{tikzpicture}
\draw (0,0) rectangle (1,1);
\draw (1,0) rectangle (2,1);
\draw (2,0) rectangle(3,1);
\draw (2,1) rectangle(3,2);
\draw (3,1) rectangle(4,2);
\draw (4,1) rectangle(5,2);
\draw (4,2) rectangle(5,3);
\node at (0.5,0.5) {1};
\node at (1.5,0.5) {2};
\node at (2.5,0.5) {3};
\node at (2.5,1.5) {4};
\node at (3.5,1.5) {5};
\node at (4.5,1.5) {6};
\node at (4.5,2.5) {7};
\node[above] at (0.5,0.9) {2};
\node[above] at (1.5,0.9) {3};
\node[below] at (1.5,0.1) {1};
\node[below] at (2.5,0.1) {2};
\node[left] at (2.1,1.5) {3};
\node[right] at (2.9,0.5) {4};
\node[below] at (3.5,1.1) {4};
\node[above] at (2.5,1.9) {5};
\node[below] at (4.5,1.1) {5};
\node[above] at (3.5,1.9) {6};
\node[left] at (4.1,2.5) {6};
\node[right] at (4.9,1.5) {7};
\draw[red,ultra thick] (0,0) -- (0,1);
\draw[red,ultra thick] (1,0) -- (2,0);
\draw[red,ultra thick] (1,1) -- (2,1);
\draw[red,ultra thick] (3,0) -- (3,1);
\draw[red,ultra thick] (2,2) -- (3,2);
\draw[red,ultra thick] (4,1) -- (5,1);
\draw[red,ultra thick] (4,2) -- (4,3);
\draw[red,ultra thick] (5,2) -- (5,3);
\end{tikzpicture}
\end{minipage}
\begin{minipage}{0.5\linewidth}
\begin{tikzpicture}
\draw (0,0) rectangle (1,1);
\draw (2,0) rectangle(3,1);
\draw (2,1) rectangle(3,2);
\draw (4,1) rectangle(5,2);
\draw (4,2) rectangle(5,3);
\node at (0.5,0.5) {1};
\node at (2.5,0.5) {3};
\node at (2.5,1.5) {4};
\node at (4.5,1.5) {6};
\node at (4.5,2.5) {7};
\node[above] at (0.5,0.9) {2};
\node[below] at (2.5,0.1) {2};
\node[left] at (2.1,1.5) {3};
\node[right] at (2.9,0.5) {4};
\node[above] at (2.5,1.9) {5};
\node[below] at (4.5,1.1) {5};
\node[left] at (4.1,2.5) {6};
\node[right] at (4.9,1.5) {7};
\draw[red,ultra thick] (0,0) -- (0,1);
\draw[red,ultra thick] (1,0) -- (1,1);
\draw[red,ultra thick] (2,0) -- (2,1);
\draw[red,ultra thick] (3,0) -- (3,1);
\draw[red,ultra thick] (2,2) -- (3,2);
\draw[red,ultra thick] (4,1) -- (5,1);
\draw[red,ultra thick] (4,2) -- (4,3);
\draw[red,ultra thick] (5,2) -- (5,3);
\end{tikzpicture}
\end{minipage}}
\end{figure}
By definition, the required perfect matching $P$ of $\mathcal{G}$ is  
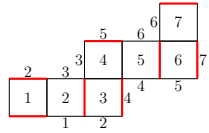
\begin{figure}[H]
\resizebox{.5\width}{.5\height}{
\begin{minipage}{0.4\linewidth}
\begin{tikzpicture}
\draw (0,0) rectangle (1,1);
\draw (1,0) rectangle (2,1);
\draw (2,0) rectangle(3,1);
\draw (2,1) rectangle(3,2);
\draw (3,1) rectangle(4,2);
\draw (4,1) rectangle(5,2);
\draw (4,2) rectangle(5,3);
\node at (0.5,0.5) {1};
\node at (1.5,0.5) {2};
\node at (2.5,0.5) {3};
\node at (2.5,1.5) {4};
\node at (3.5,1.5) {5};
\node at (4.5,1.5) {6};
\node at (4.5,2.5) {7};
\node[above] at (0.5,0.9) {2};
\node[above] at (1.5,0.9) {3};
\node[below] at (1.5,0.1) {1};
\node[below] at (2.5,0.1) {2};
\node[left] at (2.1,1.5) {3};
\node[right] at (2.9,0.5) {4};
\node[below] at (3.5,1.1) {4};
\node[above] at (2.5,1.9) {5};
\node[below] at (4.5,1.1) {5};
\node[above] at (3.5,1.9) {6};
\node[left] at (4.1,2.5) {6};
\node[right] at (4.9,1.5) {7};
\draw[red,ultra thick] (0,0) -- (1,0);
\draw[red,ultra thick] (0,1) -- (1,1);
\draw[red,ultra thick] (2,0) -- (2,1);
\draw[red,ultra thick] (3,0) -- (3,1);
\draw[red,ultra thick] (2,2) -- (3,2);
\draw[red,ultra thick] (4,1) -- (4,2);
\draw[red,ultra thick] (5,1) -- (5,2);
\draw[red,ultra thick] (4,3) -- (5,3);
\end{tikzpicture}
\end{minipage}}
\end{figure}

By the definition of $x(P|_{\mathcal{H}_t})$ for $t=1,2,3$ and Definition \ref{definition311}, we have 
\begin{align*}
x(P|_{\mathcal{H}_1}) = x_2, \quad x(P|_{\mathcal{H}_2}) = x_4x_5, \quad x(P|_{\mathcal{H}_3}) = x_7, \\
\widetilde{y}(P|_{\mathcal{H}_1}) = y_1y_2, \quad \widetilde{y}(P|_{\mathcal{H}_2}) = 1, \quad \widetilde{y}(P|_{\mathcal{H}_3}) = y_6y_7.  
\end{align*}
Using Lemma \ref{lemma312}, we have 
\begin{align*}
\frac{x(P) y(P)}{(\prod\limits_{s=1}^7 x_s) F|_{\mathbb{P}}(y_1,y_2,\ldots,y_7)} & = (\frac{x_2 y_1y_2 }{x_1x_2 (y_1y_2)|_{\mathbb{P}}})(\frac{x_4 x_5}{x_3x_4x_5 (y_3y_4y_5)|_{\mathbb{P}}})(\frac{x_7 y_6y_7}{x_6x_7 (y_6y_7)|_{\mathbb{P}}}) \\
& = \frac{x'_1 x'_3 x'_6x_8}{x_1x_3x_6}. 
\end{align*}

Therefore by Theorem \ref{theorem313}, the highest $\ell$-weight monomial and the lowest $\ell$-weight monomial in the $q$-character of $L(Y_{1,-3}Y_{3,-7}Y_{6,-2}Y_{8,-6})$ occur in 
\begin{align*}
\frac{\chi_q(\iota(x'_1 x'_3 x'_6x_8))}{\chi_q(\iota(x_1x_3x_6))}=\frac{\chi_q(L(Y_{1,-5}Y_{1,-3}) L(Y_{3,-7}Y_{3,-5})   L(Y_{6,-4}Y_{6,-2})  L(Y_{8,-6}))}{\chi_q(L(Y_{1,-5}) L(Y_{3,-5}) L(Y_{6,-4}))},
\end{align*} 
and they are 
\begin{align*}
& \frac{(Y_{1,-5}Y_{1,-3})(Y_{3,-7}Y_{3,-5}) (Y_{6,-4}Y_{6,-2})  (Y_{8,-6}) }{ (Y_{1,-5}) (Y_{3,-5}) (Y_{6,-4}) }=Y_{1,-3}Y_{3,-7}Y_{6,-2}Y_{8,-6}, \\
& \frac{(Y^{-1}_{9,5}Y^{-1}_{9,7}) (Y^{-1}_{7,3}Y^{-1}_{7,5})  (Y^{-1}_{4,6}Y^{-1}_{4,8}) (Y^{-1}_{2,4}) }{Y^{-1}_{9,5}Y^{-1}_{7,5}Y^{-1}_{4,6}} = Y^{-1}_{2,4} Y^{-1}_{4,8} Y^{-1}_{7,3}  Y^{-1}_{9,7}, 
\end{align*}
respectively.
\end{example}

\section{A new recursion for $q$-characters of Hernandez-Leclerc modules}\label{a recursive formula}

In this section, we give a recursive formula for HL-modules by an induction on the length of the highest $\ell$-weight monomials of HL-modules.

Let $d_j=\delta_{j,j_\diamond}$ for $j\in I$. We have the following theorem.
\begin{theorem}\label{theorem41}
Let $j\in I$ be a source or sink and $j>i \in I$. Then 
\begin{multline} \label{equation41}
x[\alpha_{i,j}] x[\alpha_{j+1,j+1}] =  x[\alpha_{i,j+1}] \\
+ x[\alpha_{i,\max\{i-1,j_\bullet-1\}}]^{1-\delta_{i,j_\bullet}} {x'}_{\max\{i, j_\bullet\}}^{ \min\{1,(1-\delta_{j_\bullet,i_\bullet})d_{j_\bullet-1}+\delta_{j_\bullet,i}\}} x[-\alpha_{j+2}]^{d_{j+1}} {x'}^{1-d_{j+1}}_{j+2}.
\end{multline}
\end{theorem}

\begin{proof}
By \cite[Lemma in Section 2.2]{BC19}, if there is an arrow $(j-1) \to j$ in $Q_\xi$, then after mutating the sequence $i,(i+1),\ldots,(j-1)$ we obtain the following arrows at vertices $j$ and $(j+1)$:
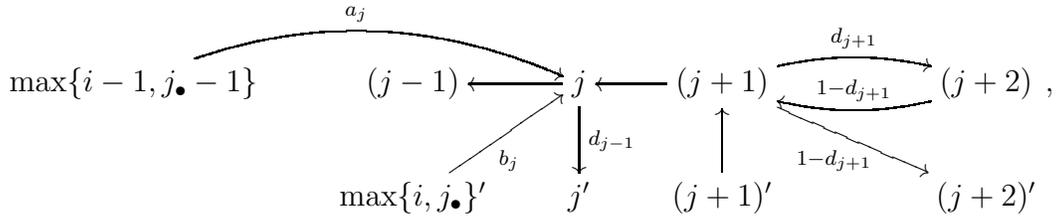
\begin{align*}
\xymatrix{
\max\{i-1,j_\bullet-1\}  \ar@/^20pt/[rr]^{a_j}  & (j-1)  & j \ar[l]  \ar[d]^{d_{j-1}}  &  (j+1) \ar[l]  \ar@/^10pt/[rr]^{d_{j+1}} \ar[drr]_{1-d_{j+1}} &&  (j+2) \ar@/^10pt/[ll]_{1-d_{j+1}}     \\
 & \max\{i,j_\bullet\}' \ar[ur]_{b_j}   & j'  & (j+1)' \ar[u]  && (j+2)' },
\end{align*}
where $a_j=1-\delta_{i,j_\bullet}$ and $b_j=\min\{1,(1-\delta_{j_\bullet,i_\bullet})d_{j_\bullet-1}+\delta_{j_\bullet,i}\}$.

We mutate vertices $j$ and $(j+1)$ in order and obtain the following arrows connecting to the vertex $j$:

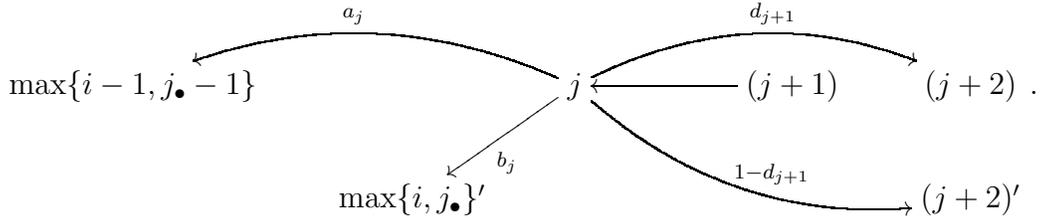
\begin{align*}
\xymatrix{
\max\{i-1,j_\bullet-1\}  &  & j  \ar@/_20pt/[ll]_{a_j}   \ar[dl]^{b_j}   \ar@/^20pt/[rrr]^{d_{j+1}}  \ar@/_20pt/[drrr]^{1-d_{j+1}} & &  (j+1)  \ar[ll]   &  (j+2)  \\
 & \max\{i,j_\bullet\}'  & &  & & (j+2)'}.
\end{align*}
Nextly, we mutate the vertex $j$ and compare the denominators in the exchange relation using Theorem \ref{cluster variables bijection with almost positive roots}, then 
\begin{align*}
x[\alpha_{i,j}] x[\alpha_{j+1,j+1}]  =  x[\alpha_{i,j+1}] + x[\alpha_{i,\max\{i-1,j_\bullet-1\}}]^{a_j} {x'}_{\max\{i, j_\bullet\}}^{b_j} x[-\alpha_{j+2}]^{d_{j+1}} {x'}^{1-d_{j+1}}_{j+2}.
\end{align*}

If there is an arrow $j\to (j-1)$ in $Q_\xi$, then we reverse all the orientations and obtain the same equation.
\end{proof}

\begin{remark}  
\label{note of our formula}
\begin{itemize} 
\item[(1)] The $q$-character of $\iota(x[\alpha_{i,j+1}])$ is the product of $q$-characters of $\iota(x[\alpha_{i,j}])$ and $\iota(x[\alpha_{j+1,j+1}])$ except those terms appearing in the product of $q$-characters of $\iota(x[\alpha_{i,\max\{i-1,j_\bullet-1\}}]^{a_j})$, $\iota({x'}_{\max\{i, j_\bullet\}}^{b_j})$, $\iota(x[-\alpha_{j+2}]^{d_{j+1}})$, and $\iota({x'}^{1-d_{j+1}}_{j+2})$. 
\item[(2)] Our recursive formula (\ref{equation41}) is different from a recursive formula given by Brito and Chari \cite[Proposition in Section 1.5]{BC19}, which are exchange relations corresponding to the mutation sequence $(i,i+1,\ldots,j+1)$, whereas our recursive formula needs to mutate $(i,i+1,\ldots,j,j+1,j)$ for a height function $\xi$ such that $j$ is source or sink.
\item[(3)] It is possible by Equation (\ref{equation41}) for us to give a combinatorial path formula in which we allow overlapping paths, generalizing Mukhin and Young's path formula for snake modules \cite{MY12a,MY12b}.
\end{itemize}
\end{remark}

In practice, given an HL-module, we construct a height function $\xi$ such that $\iota(x[\alpha_{i,j+1}])$ is an HL-module (up to isomorphic) for some $i<j$, and $j, (j+1)$ being sources or sinks, see our construction in Theorem \ref{HL defintion equivalent}. Hence Equation (\ref{equation41}) reduces to 
\begin{align*}
x[\alpha_{i,j+1}]  =  x[\alpha_{i,j}] x[\alpha_{j+1,j+1}]- x[\alpha_{i,\max\{i-1,j_\bullet-1\}}]^{a_j} {x'}_{\max\{i, j_\bullet\}}^{b_j}  x[-\alpha_{j+2}]^{d_{j+1}}
\end{align*}
giving a recursive formula for the $q$-character of an HL-module by an induction on the length of its highest $\ell$-weight monomial.

We end this section with an example illustrating Theorem \ref{theorem41}.
\begin{example}
Let  $L(Y_{1,-7}Y_{2,-4}Y_{3,-7})$ be an HL-module in type $A_4$. By Theorem \ref{HL defintion equivalent}, we take 
\[
\xi(1)=-6, \  \xi(2)=-5, \  \xi(3)=-6,  \  \xi(4)=-5, 
\]
and we have a convention that $\xi(0)=-5$, $\xi(5)=-6$. The quiver $Q_\xi$ is as follows.
\[
\xymatrix{
1  \ar[d] &   2 \ar[l] \ar[r]   &  3 \ar[d]  & 4 \ar[l]  \\
1'   &   2' \ar[u]   & 3'   & 4'  \ar[u]  }.
\]

Take $j=2$ being a source in $Q_\xi$ and by Theorem \ref{theorem41}, we have  
\begin{align*}
L(Y_{1,-7}Y_{2,-4}Y_{3,-7}) = L(Y_{1,-7}Y_{2,-4})  L(Y_{3,-7}) - L(Y_{1,-5}Y_{1,-7})L(Y_{4,-6}). 
\end{align*}
The $q$-character of $L(Y_{1,-7}Y_{2,-4}Y_{3,-7}) $ is the product of $q$-characters of $L(Y_{1,-7}Y_{2,-4})$ and $L(Y_{3,-7})$, except those terms which are in the product of $q$-characters of $L(Y_{1,-5}Y_{1,-7})$ and $L(Y_{4,-6})$. In the product of $q$-characters of $L(Y_{1,-7}Y_{2,-4})$ and $L(Y_{3,-7})$, there are 400 monomials, among them $75$ monomials are the same as the monomials in the product of $q$-characters of $L(Y_{1,-5}Y_{1,-7})$ and $L(Y_{4,-6})$.
\end{example}

\section*{Acknowledgements}
The authors are supported by the National Natural Science Foundation of China (no. 11771191, 12001254). Jian-Rong Li is supported by the Austrian Science Fund (FWF): M 2633-N32 Meitner Program. Bing Duan is grateful to Professor Ralf Schiffler for helpful discussions about snake graphs and cluster expansion formulas.

\end{document}